\documentclass[12pt]{article}
\usepackage{amsfonts}
\usepackage{mathrsfs}
\usepackage{amsthm}
\usepackage{amsmath}
\usepackage{amssymb}
\usepackage{latexsym}
\usepackage{graphicx}
\usepackage{color,varioref}
\usepackage{longtable}
\usepackage{appendix}
\usepackage{enumerate}

\usepackage{fancyhdr}
\definecolor{red}{rgb}{0.9,0,0}

\definecolor{green}{rgb}{0,0.9,0}

\definecolor{blue}{rgb}{0,0,0.9}

\newtheorem{theorem}{Theorem}[section]
\newtheorem{proposition}{Proposition}[section]
\newtheorem{lemma}{Lemma}[section]
\newtheorem{corollary}{Corollary}[section]
\newtheorem{remark}{Remark}[section]
\newtheorem{definition}{Definition}[section]
\newtheorem{assumption}{Assumption}[section]

\begin{document}

\title{\bf Perturbation analysis of a class of composite optimization problems}

\author{Peipei Tang\thanks{School of Computer and Computing Science, Hangzhou City University, Zhejiang Key Laboratory of Big Data Intelligent Computing, Hangzhou 310015, China ({\tt Email:tangpp@hzcu.edu.cn}). Research of Peipei Tang was supported by Zhejiang Provincial Natural Science Foundation of China
under Grant No. LMS26A010022.},
Chengjing Wang
\thanks{{\bf Corresponding author}, School of Mathematics, Southwest Jiaotong University, No.999, Xian Road, West Park, High-tech Zone, Chengdu 611756, China ({\tt Email:renascencewang@hotmail.com}). Research
of Chengjing Wang was supported partly by National Natural Science Foundation of China under Grant No.
U21A20169, partly by Zhejiang Provincial Natural Science Foundation
of China under Grant No. LTGY23H240002 and LMS26A010022.    } 
}

\maketitle


\begin{abstract}
In this paper, we conduct a perturbation analysis of a class of composite optimization problems, providing a unified framework for addressing both theoretical and algorithmic aspects of constrained optimization problems. By exploiting the properties of the graphical derivative,  limiting/Mordukhovich coderivative, and regular coderivative, we establish a strong second-order sufficient condition (SSOSC) for the composite optimization problem. This condition represents a novel contribution that cannot be directly inferred from existing definitions tailored to specific problems. Under certain mild assumptions on the objective function, it is demonstrated that SSOSC together with the nondegeneracy condition, the nonsingularity of Clarke's generalized Jacobian of the nonsmooth system at a Karush-Kuhn-Tucker (KKT) point, and strong regularity of the KKT point are equivalent. These findings provide a robust analytical tool for evaluating the stability of the KKT point.
\end{abstract}

\begin{keywords}
composite optimization problem, limiting/Mordukhovich coderivative,  strong regularity, strong second-order sufficient condition, second-order variational function
\end{keywords}

\textbf{MSC codes:} Primary 90C31, 90C26; Secondary 49J52, 49J53

\section{Introduction}
In this paper, we consider the following general class of composite optimization problems:
\begin{equation*}
	(P)\quad\quad\quad\min_{x\in\mathcal{R}^{n}}g(F(x)),
\end{equation*}
where $g:\mathcal{R}^{m}\rightarrow(-\infty,+\infty]$ is a proper lower semicontinuous (l.s.c.) convex function and $F:\mathcal{R}^{n}\rightarrow\mathcal{R}^{m}$ is a $\mathcal{C}^{2}$-smooth mapping. Problem $(P)$ is equivalent to the following problem
\begin{equation}\label{general-prob-equivalent}
	\min_{(x,c)\in\mathcal{R}^{n}\times\mathcal{R}}\quad \{c\, |\, (F(x),c)\in\operatorname{epi}g \}
\end{equation}
with its corresponding Lagrangian function as
\begin{equation*}
\mathcal{L}((x,c);(\mu,\gamma))=\langle\mu,F(x)\rangle+c(1+\gamma),\, (\mu,\gamma)\in\mathcal{N}_{\operatorname{epi}g}(F(x),c).
\end{equation*}
Given that the minimum of the Lagrangian function with respect to $c$ yields $-\infty$ if $\gamma\neq -1$, the dual of problem $(P)$ is formulated by
\begin{equation*}
	\max_{\mu\in\mathcal{R}^{m}}\{-g^{*}(\mu)+\min_{x\in\mathcal{R}^{n}}\langle\mu,F(x)\rangle\}.
\end{equation*}
The Karush-Kuhn-Tucker (KKT) condition for problem $(P)$ takes the following form
\begin{equation}\label{eq:comp-prob-kkt}
	F'(x)^{*}\mu=0,\quad 0\in F(x)-\partial g^{*}(\mu),
\end{equation}
which can be equivalently rewritten as the generalized equation: 
\begin{equation}\label{eq:kkt-generalized-equation}
	R(x,\mu):=\left[\begin{array}{c}
		F'(x)^{*}\mu\\
		\mu-\operatorname{Prox}_{g^{*}}(F(x)+\mu)
	\end{array}\right]=0.
\end{equation}

In recent decades, substantial research efforts, such as \cite{Bonnans2005,Mohammadi2022,Mordukhovich2018,MordukhovichNghia2015,Mordukhovich2017}, have been dedicated to analyzing the stability of composite optimization problems, where the objective function is formulated as $\varphi(x):=\varphi_{0}(x)+g(\Phi(x))$, where  $\varphi_{0}:\mathcal{R}^{n}\rightarrow\mathcal{R}$ and $\Phi:\mathcal{R}^{n}\rightarrow\mathcal{R}^{m}$ are smooth mappings. This objective function can be reformulated as the composition of $\tilde{g}(c,x):=c+g(x)$ and $F(x):=(\varphi_{0}(x),\Phi(x))$. Notably, the classical constrained optimization problem
\begin{equation}\label{classical-constrained-problem}
	\left.\begin{array}{c}\min\limits_{x\in\mathcal{R}^{n}} \varphi_{0}(x)\\
		\mbox{s.t.}\quad \Phi(x)\in\Theta,\end{array}\right.
\end{equation}
where $\Theta$ is a convex subset of $\mathcal{R}^{m}$, represents a special case of the composite optimization problem $(P)$ with $F(x):=(\varphi_{0}(x),\Phi(x))$ and $\tilde{g}(c,x):=c+\delta_{\Theta}(x)$. The proposed model $(P)$ provides a comprehensive framework that accommodates a variety of optimization problems, including linear programming, nonlinear programming, nonlinear second-order cone programming, and nonlinear semidefinite programming, among others. This unified approach significantly advances the theoretical understanding and practical applications of composite optimization.

As we know, substantial efforts have been dedicated to studying the stability analysis of solution maps in parameter-dependent optimization problems. Numerous monographs have been published that provide conditions to ensure various stability properties of these solution maps. Strong regularity, a fundamental concept that originated from the landmark work of \cite{Robinson1980}, is crucial for describing the stability of the solution mapping associated with the following generalized equation:
\begin{equation}\label{eq:generalized-equation}
    0\in\phi(x)+\mathcal{M}(x),
\end{equation}
where  $\phi:\mathcal{R}^{t}\rightarrow\mathcal{R}^{t}$ and $\mathcal{M}:\mathcal{R}^{t}\rightrightarrows\mathcal{R}^{t}$ is a set-valued mapping. Characterizing strong regularity of the canonically perturbed KKT system at the solution point of a given optimization problem has become an important topic over the past few decades. For nonlinear programming problems, Robinson \cite{Robinson1980} demonstrated that the strong second-order sufficient condition (SSOSC) combined with the linear independence constraint qualification  at the solution point serves as a sufficient condition for strong regularity of the canonically perturbed KKT system. Mordukhovich and Sarabi \cite{Mordukhovich2017} established the equivalence between full stability of local minimizers and strong regularity of the associated KKT system for a class of composite optimization problems that involve convex piecewise linear functions. For the general constrained optimization problem \eqref{classical-constrained-problem}, Mordukhovich et al. \cite{MordukhovichNghia2015} established the equivalence among strong regularity, Lipschitzian strong stability, Lipschitzian full stability, and the second-order subdifferential condition at a local minimizer, under the reducibility and nondegeneracy conditions specified in their Theorem 5.6. Additionally, as noted in \cite{Clarke1976}, Clarke's generalized Jacobian of the KKT system serves as an important tool for characterizing the stability of the KKT system at a given solution point. Characterizations of strong regularity in terms of the second-order optimality conditions and the nonsingularity of Clarke's generalized Jacobian of the KKT system for second-order cone programming problems and nonlinear semidefinite programming problems are provided in \cite{Bonnans2005} and \cite{Sun2006}, respectively. However, to the best of our knowledge, such characterizations for the general composite optimization problem $(P)$ remain unknown.

Moreover, the second-order conditions, particularly the second-order sufficient condition (SOSC) and SSOSC, play a pivotal role in establishing perturbation analysis for optimization problems. For specific instances of composite optimization problem $(P)$, such as when the function $g$ is the indicator function of the positive semidefinite cone, the sigma term can be simplified to derive an explicit expression for SOSC. This allows for the extension of the critical cone to its affine hull, thereby establishing SSOSC. However, for general composite optimization problem $(P)$, directly defining SSOSC by replacing the critical cone with its affine hull is not feasible, as the domain of the sigma term in SOSC is restricted to the critical cone. Formulating an expression for SSOSC that naturally extends SOSC while enabling meaningful analysis presents significant challenges. In this paper, we propose a definition of SSOSC (as detailed in Section \ref{sec:SRCQ-SOSC-SSOSC-nondegeneracy}) for composite optimization problem $(P)$. We also explore fundamental properties related to function $g$ and its proximal operator $\operatorname{Prox}_{g}$. Based on the defined SSOSC and a set of mild assumptions—including  the $\mathcal{C}^{2}$-cone reducibility of $g$, and other verifiable conditions—we demonstrate the equivalence among the following conditions: the combination of SSOSC and the nondegeneracy condition, the nonsingularity of Clarke's generalized Jacobian of the mapping $R$ at a solution $(x,\mu)$ of the KKT system \eqref{eq:comp-prob-kkt}, and strong regularity of the KKT point $(x,\mu)$. Notably, the class of $\mathcal{C}^{2}$-cone reducible functions is extensive, encompassing all indicator functions of $\mathcal{C}^{2}$-cone reducible sets, the nuclear norm function and the matrix spectral function.

The methodology employed in this paper draws upon established frameworks from \cite{Bonnans2000, DontchevandRockafellar2014, Mordukhovich2006, Mordukhovich2018, Rockafellar1998}. It systematically incorporates the theoretical underpinnings of set-valued mapping derivatives, particularly the graphical derivative, limiting/Mordukhovich and regular coderivatives. Building on these foundational concepts, which have been extensively validated in variational analysis, optimization theory, and control systems, this paper establishes a robust framework for defining SSOSC tailored for composite optimization problems of the form $(P)$.

The remaining sections of this paper are organized as follows. Section \ref{sec:preliminaries} reviews some fundamental concepts from variational analysis and generalized differentiation, including the second-order subderivative and derivatives of set-valued mappings that are extensively employed throughout the paper. Section \ref{sec:prop-g-prox_g} establishes essential properties of function $g$ and its associated proximal mapping, which are instrumental for subsequent analysis. Section \ref{sec:SRCQ-SOSC-SSOSC-nondegeneracy} introduces SOSC, SSOSC and constraint qualification conditions, with particular emphasis on the definition of  SSOSC, which serves as the foundation for establishing the key equivalence conditions in perturbation analysis. Section \ref{sec:equiv-nonsignularity-SSOSC} presents the equivalent conditions for the nonsingularity of Clarke's generalized Jacobian of the nonsmooth system at a KKT point. Finally, Section \ref{sec:Conclusion} presents our conclusions and future research.

Our notation is mostly standard. For a given point $x\in\mathcal{R}^{m}$ and $\varepsilon>0$, the closed ball centered at $x$ with radius $\varepsilon$ is defined as $\mathbb{B}(x,\varepsilon):=\{u\in\mathcal{R}^{m}\, |\, \|u-x\|\leq\varepsilon\}$. Given a set $C\subseteq\mathcal{R}^{m}$,  the indicator function $\delta_{C}$ of the set $C$ is defined by $\delta_{C}(x)=0$ if $x\in C$, and $\delta_{C}(x)=+\infty$ otherwise. For a point $x\in\mathcal{R}^{m}$, $\operatorname{dist}(x,C)$ represents the distance from $x$ to the set $C$, while $\operatorname{\Pi}_{C}(x)$ denotes the projection of $x$ onto $C$. For a sequence $\{x^{k}\}$,  $x^{k}\xrightarrow{C}x$ means that $x^{k}\rightarrow x$ with $x^{k}\in C$. The smallest cone containing $C$ is called the positive hull of $C$, defined by $\operatorname{pos}C=\{0\}\cup\{\lambda x\, |\, x\in C,\, \lambda>0\}$. If $C=\emptyset$, then $\operatorname{pos}C=\{0\}$ and if $C\neq\emptyset$, then $\operatorname{pos}C=\{\lambda x\, |\, x\in C,\, \lambda\geq0\}$. The convex hull of $C$, denoted by $\operatorname{conv}C$, is the intersection of all convex sets containing $C$. For a nonempty closed convex cone $C$, the polar cone $C^\circ$ is defined as $C^{\circ}=\{y\,|\, \langle x,y\rangle\leq 0,\,\forall\, x\in C\}$.  The lineality space of $C$ is  $\operatorname{lin}C:=C\cap(-C)$, representing the largest subspace contained in $C$, and the affine space $\operatorname{aff}C:=C-C$ is the smallest subspace containing $C$. If $C$ is a nonempty closed convex subset of $\mathcal{R}^{m}$ with $0\in C$, the gauge of $C$ is the function $\gamma_{C}$ defined by $\gamma_{C}(x):=\inf\{\lambda\geq0\,|\,x\in\lambda C\}$. For a given matrix $A\in\mathcal{R}^{m\times n}$, the null space of $A$ is $\operatorname{ker}A:=\{x\in\mathcal{R}^{n}\,|\,Ax=0\}$.

\section{Preliminaries}\label{sec:preliminaries}

In this section, we introduce fundamental concepts of variational analysis and generalized differentiation that will be extensively used in the remainder of the paper. We adopt the standard notation and terminology consistent with the existing literature, and refer readers to the monographs \cite{Bonnans2000, Mordukhovich2006, Mordukhovich2018, Mordukhovich2024, Rockafellar1998} for a comprehensive treatment of these topics.

Consider a real-valued function $r:\mathcal{R}^{m}\rightarrow\overline{\mathcal{R}}$. Its epigraph $\operatorname{epi}r$ is defined as
$
\operatorname{epi}r:=\{(x,c)\, |\, x\in\operatorname{dom}r,\ c\in\mathcal{R},\ r(x)\leq c\}$, 
where $\operatorname{dom}r:=\{x\in\mathcal{R}^{m}\,|\,r(x)<+\infty\}$. 
The conjugate function $r^{*}$ at $x\in\mathcal{R}^{m}$ is given by
$
r^{*}(x):=\sup\limits_{u\in\operatorname{dom}r}\{\langle x,u\rangle-r(u)\}$.
The extended function $r$ is said to be proper if its epigraph is nonempty and lacks vertical lines. This occurs when there exists at least one $x\in\mathcal{R}^{m}$ such that $r(x)<+\infty$ and $r(x)>-\infty$ for all $x\in\mathcal{R}^{m}$. Equivalently, $r$ is proper if and only if $\operatorname{dom}r$ is nonempty and $r$ is finite on its domain; otherwise, $r$ is improper. For a proper l.s.c. function $r$, recall the proximal mapping $\operatorname{Prox}_{\sigma r}$ defined by
\begin{equation*}
	\begin{aligned}	\operatorname{Prox}_{\sigma r}(x)&:=\mathop{\operatorname{argmin}}_{u\in\mathcal{R}^{m}}\{r(u)+\frac{1}{2\sigma}\|u-x\|^{2}\},
	\end{aligned}
\end{equation*}
where $\sigma>0$ is a parameter.
If $r$ is a proper l.s.c. convex function, the proximal mapping $\operatorname{Prox}_{\sigma r}$ is single-valued, Lipschitz continuous, and satisfies the Moreau identity (see e.g., \cite[Theorem~2.26]{Rockafellar1998},  \cite[Theorem~31.5]{Rockafellar1970}) 
\begin{equation}\label{eq:moreau-identity}
\operatorname{Prox}_{\sigma r}(x)+\sigma\operatorname{Prox}_{\sigma^{-1}r^{*}}(\sigma^{-1}x)=x,\; \forall\, x\in\mathcal{R}^{m}.
\end{equation}

For the extended real-valued function $r$ and $x\in\operatorname{dom}r$, the regular/Fr\'{e}chet subdifferential (also known  as the presubdifferential or viscosity
subdifferential) of $r$ at $x$ is defined as the set 
\begin{equation*}
	\widehat{\partial}r(x):=\{v\in\mathcal{R}^{m}\,|\, r(x')\geq r(x)+\langle v,x'-x\rangle+o(\|x'-x\|)\}.
\end{equation*}
The limiting/Mordukhovich subdifferential (also referred to as the basic or general subdifferential) of $r$ at $x$, denoted by $\partial r(x)$, is given by
\begin{equation*}
	\partial r(x):=\{v\in\mathcal{R}^{m}\,|\, \exists\, x^{k}\rightarrow x\ \mbox{with}\ r(x^{k})\rightarrow r(x),\, v^{k}\in\widehat{\partial}r(x^{k})\ \mbox{with}\ v^{k}\rightarrow v\}.
\end{equation*}
The singular subdifferential (also known as the horizon subdifferential) of $r$ at $x$ is defined by
\begin{equation*}
	\partial^{\infty}r(x):=\{v\in\mathcal{R}^{m}\,|\,\exists\, t_{k}\downarrow 0,\ x^{k}\rightarrow x\ \mbox{with}\, r(x^{k})\rightarrow r(x),\,v^{k}\in\widehat{\partial}r(x^{k}),\, t_{k}v^{k}\rightarrow v\}.
\end{equation*}
According to \cite[Theorem~1.93]{Mordukhovich2006}, if $r$ is a convex function, the regular/Fr\'{e}chet subdifferential and the limiting/Mordukhovich subdifferential of $r$ at $x$ coincide with the set of subgradients of $r$ at $x$ in the sense of convex analysis. A function $r$ is said to be subdifferentiable at $x$ if $\partial r(x)\neq\emptyset$.

Let $x\in\mathcal{R}^{m}$ be a point where $r(x)$ is finite. The lower and upper directional epiderivatives of $r$ are defined as follows:
\begin{equation*}
	\begin{aligned}
	r^{\downarrow}_{-}(x,d):=\liminf_{{t\downarrow 0}\atop{d'\rightarrow d}}\frac{r(x+td')-r(x)}{t},\
	r^{\downarrow}_{+}(x,d):=\sup_{\{t_{n}\}\in\operatorname{\Sigma}}\liminf_{{n\rightarrow\infty}\atop{d'\rightarrow d}}\frac{r(x+t_{n}d')-r(x)}{t_{n}},
	\end{aligned}
\end{equation*}
where $\operatorname{\Sigma}$ denotes the set of positive real sequences $\{t_{n}\}$ converging to zero. The function $r$ is said to be directionally epi-differentiable at $x$ in a direction $d$ if $r^{\downarrow}_{-}(x,d)=r^{\downarrow}_{+}(x,d)$. Every convex function is directionally epi-differentiable at all points of its domain. Assuming $r(x)$ and the respective directional epiderivatives $r^{\downarrow}_{-}(x,d)$ and $r^{\downarrow}_{+}(x,d)$ are finite, define the lower and upper second-order epiderivatives as
\begin{equation*}
	\begin{aligned}
	r^{\downdownarrows}_{-}(x;d,w)&:=\liminf_{{t\downarrow 0}\atop{w'\rightarrow w}}\frac{r(x+td+\frac{1}{2}t^{2}w')-r(x)-tr^{\downarrow}_{-}(x,d)}{\frac{1}{2}t^{2}},\\
	r^{\downdownarrows}_{+}(x;d,w)&:=\sup_{\{t_{n}\}\in\operatorname{\Sigma}}\liminf_{{n\rightarrow\infty}\atop{w'\rightarrow w}}\frac{r(x+t_{n}d+\frac{1}{2}t_{n}^{2}w')-r(x)-t_{n}r^{\downarrow}_{+}(x,d)}{\frac{1}{2}t_{n}^{2}},
	\end{aligned}
\end{equation*}
respectively. The function $r$ is parabolically epi-differentiable at $x$ in the direction $d$ if $r^{\downdownarrows}_{-}(x;d,\cdot)=r^{\downdownarrows}_{+}(x;d,\cdot)$.
Recall that the second-order subderivative of $r$ at a feasible point $x\in\operatorname{dom}r$ for $u\in\partial r(x)$ is defined by
\begin{equation*}
	\begin{aligned}
	d^{2}r(x,u)(d):=\liminf_{{t\downarrow 0}\atop{d'\rightarrow d}}\Delta_{t}^{2}r(x,u)(d'),\ 
	\Delta_{\tau}^{2}r(x,u)(d):=\frac{r(x+\tau d)-r(x)-\tau\langle u,d\rangle}{\frac{1}{2}\tau^{2}}.
	\end{aligned}
\end{equation*}
The function $r$ is called twice epi-differentiable at $x$ for $u$ if $\Delta_{t}^{2}r(x,u)$ epi-converges to $d^{2}r(x,u)$ as $t\downarrow 0$.
The second-order subderivative $d^{2}r(x,u)$ is l.s.c., as it arises from the lower epi-limit of the corresponding ratio function. By \cite[Proposition~13.5]{Rockafellar1998}, it is positively homogeneous of degree 2 with $\operatorname{dom}d^{2}r(x,u)\subseteq\{d\, |\, r^{\downarrow}_{-}(x,d)=\langle u,d\rangle\}$. Furthermore, if $r$ is a proper l.s.c. convex function, then $\frac{1}{2}d^{2}r(x,u)$ is convex and proper, and its conjugate function is $\frac{1}{2}d^{2}r^{*}(u,x)$, when $x\in\operatorname{dom}r$ and $u\in\partial r(x)$. One may see \cite[Theorems~2.2 and~2.4]{Rockafellar1990} for more details.

\begin{definition}
	Let $r:\mathcal{R}^{m}\rightarrow\overline{\mathcal{R}}$ be an extended real-valued function, $x\in\operatorname{dom}r$, and $u\in\partial r(x)$. We say that $r$ is parabolically regular at $x$ in the direction $d$ for $u$, if 
	\begin{equation*}
	\inf\limits_{w\in\mathcal{R}^{m}}\{r^{\downdownarrows}_{-}(x;d,w)-\langle u,w\rangle\}=d^{2}r(x,u)(d).
	\end{equation*}
\end{definition}

Given a set $C\subseteq\mathcal{R}^{m}$ and a point $x\in C$,
the (Bouligand) tangent/contingent cone to $C$ at $x$ is defined as
\begin{equation*}
	\mathcal{T}_{C}(x):=\limsup_{t\downarrow 0}\frac{C-x}{t}=\{y\in\mathcal{R}^{m}\, |\, \exists\ t_{k}\downarrow 0,\ y^{k}\rightarrow y\ \mbox{with}\ x+t_{k}y^{k}\in C\}.
\end{equation*}
The corresponding regular tangent cone to $C$ at $x$ and the second-order tangent set to $C$ at $x$ for a vector $w\in\mathcal{T}_{C}(x)$ are given by
\begin{equation*}
	\widehat{\mathcal{T}}_{C}(x):=\liminf_{{ x'\xrightarrow{C}x}\atop{t\downarrow 0}}\frac{C-x'}{t}\;\mbox{and}\;\mathcal{T}_{C}^{2}(x,w):=\limsup_{t\downarrow 0}\frac{C-x-tw}{\frac{1}{2}t^{2}},
\end{equation*}
respectively. 
The regular/Fr\'{e}chet normal cone to $C$ at $x$ admits the equivalent characterization 
\begin{equation*}
	\widehat{\mathcal{N}}_{C}(x):=\left\{y\in\mathcal{R}^{m}\, \left|\, \limsup_{x'\xrightarrow{C} x}\frac{\langle y,x'-x\rangle}{\|x'-x\|}\leq 0\right.\right\}=\mathcal{T}_{C}^{\circ}(x).
\end{equation*}
The limiting/Mordukhovich normal cone to $C$ at $x$ is defined by
\begin{equation*}
	\mathcal{N}_{C}(x):=\limsup_{x'\xrightarrow{C} x}\widehat{\mathcal{N}}_{C}(x'),
\end{equation*}
which is equivalent to Mordukhovich's original definition \cite{MORDUKHOVICH1976}, namely, 
\begin{equation*}
	\mathcal{N}_{C}(x):=\limsup_{x'\rightarrow x}\{\operatorname{pos}(x'-\operatorname{\Pi}_{C}(x'))\},\; \mathcal{N}_{C}(x)^{\circ}=\widehat{\mathcal{T}}_{C}(x),
\end{equation*}
provided that $C$ is locally closed at $x$.
When $C$ is convex, both the tangent cone $\mathcal{T}_{C}$ and the regular tangent cone $\widehat{\mathcal{T}}_{C}$ coincide with the classical tangent cone, while both the regular/Fr\'{e}chet normal cone $\widehat{\mathcal{N}}_{C}$ and the limiting/Mordukhovich normal cone $\mathcal{N}_{C}$ reduce to the classical normal cone of convex analysis. The notions $\mathcal{T}_{C}(x)$ and $\mathcal{N}_{C}(x)$ are commonly used to denote these cones, respectively.

\begin{definition}
	Let $C\subseteq\mathcal{R}^{m}$ and $K\subseteq\mathcal{R}^{t}$ be convex closed sets. We say that the set $C$ is $\mathcal{C}^{\ell}$-reducible to the set $K$, at a point $x\in C$, if there exist a neighborhood $N$ at $x$ and an $\ell$-times continuously differentiable mapping $\Xi:N\rightarrow\mathcal{R}^{t}$ such that
	$\Xi'(x):\mathcal{R}^{m}\rightarrow\mathcal{R}^{t}$ is onto and $C\cap N=\{x\in N\, |\, \Xi(x)\in K\}$.	We say that the reduction is pointed if the tangent cone $\mathcal{T}_{K}(\Xi(x))$ is a pointed cone. If, in addition, the set $K-\Xi(x)$ is a pointed closed convex cone, we say that $C$ is $\mathcal{C}^{\ell}$-cone reducible at $x$. We can assume without loss of generality that $\Xi(x)=0$.
\end{definition}	

We say a closed proper convex function $r:\mathcal{R}^{m}\rightarrow(-\infty,+\infty]$ is $\mathcal{C}^{2}$-cone reducible at $x$, if the set $\operatorname{epi}r$ is $\mathcal{C}^{2}$-cone reducible at $(x,r(x))$. Furthermore, $r$ is said to be $\mathcal{C}^{2}$-cone reducible if it is $\mathcal{C}^{2}$-cone reducible at every $x\in\operatorname{dom}r$. 


For a set-valued mapping $S:\mathcal{R}^{m}\rightrightarrows\mathcal{R}^{n}$, the range of $S$ is defined as the set $\operatorname{rge}S:=\{u\in\mathcal{R}^{n}\,|\,\exists\,x\in\mathcal{R}^{m}\ \mbox{with}\ u\in S(x)\}$, while the domain of $S$ is given by $\operatorname{dom}S:=\{x\in\mathcal{R}^{m}\,|\, S(x)\neq\emptyset\}$. The inverse mapping $S^{-1}:\mathcal{R}^{n}\rightrightarrows\mathcal{R}^{m}$ is defined by $S^{-1}(u):=\{x\in\mathcal{R}^{m}\,|\,u\in S(x)\}$. Specifically, given a point  $x\in\operatorname{dom}S$, the graphical derivative of $S$ at $x$ for any $u\in S(x)$ is the mapping $DS(x,u):\mathcal{R}^{m}\rightrightarrows\mathcal{R}^{n}$ defined by
\begin{equation*}
	z\in DS(x,u)(w)\ \Longleftrightarrow\ (w,z)\in\mathcal{T}_{\operatorname{gph}S}(x,u),
\end{equation*}
where $\operatorname{gph}S$ is the graph of $S$, a subset of $\mathcal{R}^{m}\times\mathcal{R}^{n}$, notably $\operatorname{gph}S:=\{(x,y)\, |\, y\in S(x)\}$. The limiting/Mordukhovich coderivative mapping $D^{*}S(x,u):\mathcal{R}^{n}\rightrightarrows\mathcal{R}^{m}$ is defined by
\begin{equation*}
	v\in D^{*}S(x,u)(y)\ \Longleftrightarrow\ (v,-y)\in\mathcal{N}_{\operatorname{gph}S}(x,u).
\end{equation*}
Similarly, the regular coderivative $\widehat{D}^{*}S(x,u):\mathcal{R}^{n}\rightrightarrows\mathcal{R}^{m}$ is defined by
\begin{equation*}
	v\in \widehat{D}^{*}S(x,u)(y)\ \Longleftrightarrow\ (v,-y)\in\widehat{\mathcal{N}}_{\operatorname{gph}S}(x,u).
\end{equation*}
When $S$ is single-valued at $x$, i.e., $S(x)=\{u\}$, the notations $DS(x,u)$, $D^{*}S(x,u)$ and $\widehat{D}^{*}S(x,u)$ are simplified to $DS(x)$, $D^{*}S(x)$ and $\widehat{D}^{*}S(x)$, respectively.  Moreover, if $S$ is single-valued and locally Lipschitz continuous at $x$, then $D^{*}S(x)(d)=\partial\langle d,S\rangle(x)$ and $\widehat{D}^{*}S(x)(d)=\widehat{\partial}\langle d,S\rangle(x)$ as stated in  \cite[Proposition~9.24]{Rockafellar1998}. One may see for example, \cite[Proposition~8]{Ioffe1984ApproximateSA} for additional details. 
The graphical derivative mapping $DS(x,u)$ and the limiting/Mordukhovich coderivative mapping $D^{*}S(x,u)$ exhibit simple behaviors when taking the inverse of the mapping $S$:
\begin{align}
	z\in DS(x,u)(w)\,&\Longleftrightarrow\, w\in D(S^{-1})(u,x)(z),\label{eq:inverse-of-DS}\\
	v\in D^{*}S(x,u)(y)\,&\Longleftrightarrow\, -y\in D^{*}(S^{-1})(u,x)(-v).
	\label{eq:inverse-of-DS&DS*}
\end{align}

\begin{definition}
	The multifunction $S:\mathcal{R}^{m}\rightrightarrows\mathcal{R}^{n}$ is said to be metrically subregular at $\bar{z}$ for $\bar{w}$ if there
	exists $\kappa\geq0$ along with $\varepsilon>0$ and $\delta>0$ such that for all  $z\in\mathbb{B}(\bar{z},\varepsilon)$,
	\begin{equation*}
		\operatorname{dist}(z,S^{-1}(\bar{w}))\leq\kappa\operatorname{dist}(\bar{w},S(z)\cap\mathbb{B}(\bar{w},\delta)).
	\end{equation*}
\end{definition}

For a vector-valued locally Lipschitz continuous function $G:\mathcal{O}\rightarrow\mathcal{R}^{n}$, where $\mathcal{O}$ is an open subset of $\mathcal{R}^{m}$, $G$ is F(r\'{e}chet)-differentiable almost everywhere on $\mathcal{O}$ by Rademacher's theorem \cite[Theorem~9.60]{Rockafellar1998}. Let $\mathcal{D}_{F}$ denote the set of all points where $G$ is F-differentiable and  $G'(x)\in\mathcal{R}^{n\times m}$ be the Jacobian of $G$ at $x\in \mathcal{D}_{F}$, whose adjoint is denoted as $G'(\bar{x})^{*}$. The B-subdifferential of $G$ at $x\in\mathcal{O}$ is defined by
\begin{equation*}
	\partial_{B}G(x):=\{U\in\mathcal{R}^{n\times m}\,|\, \exists\, x^{k}\xrightarrow{\mathcal{D}_{F}} x, G'(x^{k})\rightarrow U\}.
\end{equation*}
The corresponding Clarke subdifferential of $G$ at $x$ is defined by $\partial G(x):=\operatorname{conv}\partial_{B}G(x)$. For a given $\mathcal{C}^{2}$-smooth mapping $G:\mathcal{R}^{m}\rightarrow\mathcal{R}^{n}$, denote $G''(x):\mathcal{R}^{m}\rightarrow\mathcal{R}^{n\times m}$ the second-order G\^{a}teaux derivative of $G$ at $x$.

\begin{definition}
	We say a function $G:\mathcal{R}^{m}\rightarrow\mathcal{R}^{n}$ is directionally differentiable at $x\in\mathcal{R}^{m}$ in a direction $d\in\mathcal{R}^{m}$, if
	\begin{equation*}
	G'(x,d):=\lim\limits_{t\downarrow 0}\frac{G(x+td)-G(x)}{t}
	\end{equation*}
	exists. If $G$ is directionally differentiable at $x$ in every direction $d\in\mathcal{R}^{m}$, $G$ is said to be directionally differentiable at $x$.
\end{definition}	
 It follows from \cite[Proposition~2.2]{Hoheisel2012} that $DG(x)(d)=\{G'(x,d)\}$ for all $d\in\mathcal{R}^{m}$, provided that $G$ is directionally differentiable at $x$. For a proper l.s.c. convex function $r$, given $x\in\operatorname{dom}r$ and $u\in\partial r(x)$, if $r$ is twice epi-differenitiable at $x$ for $u$, by  \cite[Theorem~13.40]{Rockafellar1998} and \eqref{eq:inverse-of-DS}, the proximal mapping $\operatorname{Prox}_{r}$ is directionally differentiable at $x+u$ with its directional derivative given by  $\operatorname{Prox}_{r}'(x+u,\cdot)=\operatorname{Prox}_{\frac{1}{2}d^{2}r(x,u)}(\cdot)$. 

To facilitate our analysis, we impose the following assumptions, which are crucial for establishing our main theoretical results.
\begin{assumption}\label{assump-1}
	Suppose that $r:\mathcal{R}^{m}\rightarrow(-\infty,+\infty]$ is a proper l.s.c. convex $\mathcal{C}^{2}$-cone reducible function.
\end{assumption}	
\begin{assumption}\label{assumption-2}
	Let $r:\mathcal{R}^{m}\rightarrow(-\infty,+\infty]$ be a proper l.s.c. convex function with $\bar{x}\in\operatorname{dom}r$ and $\bar{u}\in\partial r(\bar{x})$. Suppose that
	\begin{equation*}
		\begin{aligned}
		\operatorname{dom}D^{*}(\partial r)(\bar{x},\bar{u})&=\operatorname{aff}\{d\, |\, r^{\downarrow}_{-}(\bar{x},d)=\langle\bar{u},d\rangle\},\\
		\{d\, |\, 0\in D^{*}\operatorname{Prox}_{r}(\bar{x}+\bar{u})(d)\}&=\operatorname{aff}\{d\, |\, 0\in D\operatorname{Prox}_{r}(\bar{x}+\bar{u})(d)\}.
		\end{aligned}
	\end{equation*}
\end{assumption}	
\begin{assumption}\label{assumption-3}
	Let $r:\mathcal{R}^{m}\rightarrow(-\infty,+\infty]$ be a proper l.s.c. convex function with $\bar{x}\in\operatorname{dom}r$ and $\bar{u}\in\partial r(\bar{x})$. For any $v\in\operatorname{rge}(D^{*}\operatorname{Prox}_r(\bar{x}+\bar{u}))$, suppose that
	\begin{equation*}
		\mathop{\operatorname{argmin}}_{{d\in\mathcal{R}^{m},v=Ud}\atop{U\in\partial \operatorname{Prox}_{r}(\bar{x}+\bar{u})}}\langle v,d-v\rangle\cap\{d\,|\, v\in D^{*}\operatorname{Prox}_{r}(\bar{x}+\bar{u})(d)\}\neq\emptyset.
	\end{equation*}
\end{assumption}

We say that a function $r$ satisfies Assumption \ref{assumption-2}, or Assumption \ref{assumption-3} at the point $\bar{x}$ for $\bar{u}$. For brevity, we refer to a function $r$ as satisfying Assumption \ref{assumption-2} or Assumption \ref{assumption-3} when the context is clear. These assumptions are satisfied by various functions arsing in conic programming, including the $\ell_{p}$ norm functions for $p = 1, 2, \infty$, and indicator functions of standard cones such as the positive orthant, second-order, and positive semidefinite cones. In Appendix \ref{Appendix-A}, we establish some sufficient conditions for Assumptions \ref{assumption-2} and \ref{assumption-3}. Meanwhile, Appendix \ref{Appendix-B} provides the verification of these two assumptions for the indicator function of the positive semidefinite cone.

\section{Fundamental properties}\label{sec:prop-g-prox_g}	
In this section, for a given proper l.s.c convex function $r$, we establish a useful property of the Clarke 
subdifferential of the proximal mapping  $\operatorname{Prox}_{r}$ at the point of interest. Additionally, leveraging the properties of the graphical derivative, we also present additional characteristics associated with the second-order subderivative and proximal mapping. 

Given $\bar{x}\in\operatorname{dom
}r$ and $\bar{u}\in\partial r(\bar{x})$, the subsequent proposition provides full characterizations for   $\operatorname{dom}d^{2}(\bar{x},\bar{u})$ and $\operatorname{dom}D(\partial r)(\bar{x},\bar{u})$.
\begin{proposition}\label{prop-domd2g}
	Let $r:\mathcal{R}^{m}\rightarrow (-\infty,+\infty]$ be a proper l.s.c. convex function with $\bar{x}\in\operatorname{dom}r$ and $\bar{u}\in\partial r(\bar{x})$. One has
	\begin{equation}
	\label{eq:domain-graphical-derivative-partial-r}	\operatorname{dom}D(\partial r)(\bar{x},\bar{u})=\operatorname{dom}d^{2}r(\bar{x},\bar{u}).
	\end{equation}
	Furthermore, when $r$ is additionally $\mathcal{C}^{2}$-cone reducible, the function $r$ is parabolically epi-differentiable at $\bar{x}$ for every $w\in\{d\, |\, r^{\downarrow}_{-}(\bar{x},d)=\langle \bar{u},d\rangle\}$ and parabolically regular at $\bar{x}$ for $\bar{u}$. Moreover, $r$ is twice epi-differentiable at $\bar{x}$ for $\bar{u}$  with
	\begin{equation*}
		\operatorname{dom}d^{2}r(\bar{x},\bar{u})=\{d\, |\, r^{\downarrow}_{-}(\bar{x},d)=\langle \bar{u},d\rangle\}.
	\end{equation*}
\end{proposition}
\begin{proof} 
    To establish the equality in \eqref{eq:domain-graphical-derivative-partial-r}, we first show that the second-order subderivative $d^{2}r(\bar{x},\bar{u})$ is subdifferentiable at every point of its domain. As shown in  \cite[Proposition~13.20]{Rockafellar1998}, $d^{2}r(\bar{x},\bar{u})$ is convex and equal to $\gamma_{C}^{2}$, where $C=\{v\,|\,d^{2}r(\bar{x},\bar{u})(v)\leq 1\}$ is a closed convex set containing $0$. Since $\gamma_{C}$ is a positively homogeneous convex function with $\gamma_{C}(0)=0$, it follows from  \cite[Proposition~2.124]{Bonnans2000} that $\gamma_{C}$ is subdifferentiable at any $v\in\operatorname{dom}d^{2}r(\bar{x},\bar{u})$. For any $v\in\operatorname{dom}d^{2}r(\bar{x},\bar{u})$ and $v\notin\mathop{\operatorname{argmin}}\limits_{d\in\mathcal{R}^{m}}\{d^{2}r(\bar{x},\bar{u})(d)\}$, let $w\in\partial\gamma_{C}(v)$, where $w$ is finite and satisfies $\gamma_{C}(y)\geq\gamma_{C}(v)+\langle w,y-v\rangle,\;\forall\; y\in\mathcal{R}^{m}$. 
	By \cite[Propositions~13.5 and~13.20]{Rockafellar1998}, $d^{2}r(\bar{x},\bar{u})(v)=\gamma^{2}_{C}(v)>0$ and there exists a neighborhood $\mathcal{O}$ of $v$ such that for any $y\in\mathcal{O}$, both $d^{2}r(\bar{x},\bar{u})(y)>0$ and $\gamma_{C}(v)+\langle w,y-v\rangle>0$ hold. Therefore, the square of $\gamma_{C}(y)$ satisfies the following relationship
	\begin{equation*}
		\begin{aligned}
		d^{2}r(\bar{x},\bar{u})(y)=\gamma_{C}^{2}(y)&\geq\gamma_{C}^{2}(v)+2\gamma_{C}(v)\langle w,y-v\rangle+\langle w,y-v\rangle^{2}\\
		&\geq d^{2}r(\bar{x},\bar{u})(v)-2\gamma_{C}(v)\|w\|\|y-v\|.
		\end{aligned}
	\end{equation*}
	Finally, combining \cite[Proposition~2.1]{Mohammadi2022}, it concludes $\partial d^{2}r(\bar{x},\bar{u})(v)\neq\emptyset$. If, however, $v\in\mathop{\operatorname{argmin}}\limits_{d\in\mathcal{R}^{m}}\{d^{2}r(\bar{x},\bar{u})(d)\}$, then $0\in\partial d^{2}r(\bar{x},\bar{u})(v)$ and \eqref{eq:domain-graphical-derivative-partial-r} follows from \cite[Theorem~13.40]{Rockafellar1998}.

	When $r$ is $\mathcal{C}^{2}$-cone reducible, as established in  \cite[Theorem~6.2]{Mohammadi2021} and  \cite[Example~13.62(b)]{Rockafellar1998}, the function $r$ is parabolically epi-differentiable at $\bar{x}$ for all $w\in\{d\, |\, r^{\downarrow}_{-}(\bar{x},d)=\langle \bar{u},d\rangle\}$. Furthermore, \cite[Propositions~3.103 and~3.136]{Bonnans2000} demonstrates that $r$ is also parabolically
	regular at $\bar{x}$ for $\bar{u}$. Combining \cite[Theorem~3.8]{Mohammadi2020},  the function $r$ is properly twice epi-differentiable at $\bar{x}$ for $\bar{u}$,
	with the domain $\operatorname{dom}d^{2}r(\bar{x},\bar{u})$ defined as $\{d\, |\, r^{\downarrow}_{-}(\bar{x},d)=\langle \bar{u},d\rangle\}$. This completes the proof.
\end{proof}

\begin{proposition}\label{prop-partial-proximal-mapping}
	Let $r:\mathcal{R}^{m}\rightarrow (-\infty,+\infty]$ be a proper l.s.c. convex function. For any $\bar{x}\in\operatorname{dom}r$, $\bar{u}\in\partial r(\bar{x})$, $U\in\partial\operatorname{Prox}_{r}(\bar{x}+\bar{u})$ and $d\in\mathcal{R}^{m}$, we have
	$
	\langle Ud,d-Ud\rangle\geq0.
	$
\end{proposition}
\begin{proof}
	For the given function $r$, the limiting/Mordukhovich subdifferential $\partial r$ is maximal monotone \cite{Rockafellar70}. As established by \cite[Theorem~2.1]{Poliquin1998}, $\langle v,d\rangle\geq0$ for any $v\in D^{*}(\partial r)(\bar{x},\bar{u})(d)$. By invoking \eqref{eq:inverse-of-DS&DS*} and  \cite[Theorem~1.62]{Mordukhovich2006}, $v\in D^{*}\operatorname{Prox}_{r}(\bar{x}+\bar{u})(d)$ is equivalent to $v-d\in D^{*}(\partial r)(\bar{x},\bar{u})(-v)$. Consequently, the following inequality is valid
	\begin{equation}\label{eq:prop-partial-proximal-mapping-1}
		\langle v,d-v\rangle\geq0, \, \forall\, v\in D^{*}\operatorname{Prox}_{r}(\bar{x}+\bar{u})(d).
	\end{equation}
	Observing from \cite[Proposition~2.4]{Gfrerer2022}, for any $V\in\partial_{B}\operatorname{Prox}_{r}(\bar{x}+\bar{u})$ and $d\in\mathcal{R}^{m}$, the element $Vd$ belongs to $D^{*}\operatorname{Prox}_{r}(\bar{x}+\bar{u})(d)$. Substituting this into  \eqref{eq:prop-partial-proximal-mapping-1}, it follows that  $\langle Vd,d\rangle\geq\|Vd\|^{2}$ for any $V\in\partial_{B}\operatorname{Prox}_{r}(\bar{x}+\bar{u})$ and $d\in\mathcal{R}^{m}$. By applying Carath\'{e}odory's theorem (see e.g., \cite[Theorem~17.1]{Rockafellar1970}), for any $U\in\partial\operatorname{Prox}_{r}(\bar{x}+\bar{u})$, there exist a positive integer $t$, matrices  $U_{i}\in\partial_{B}\operatorname{Prox}_{r}(\bar{x}+\bar{u})$,  $i=1,2,\ldots,t$ and  real numbers $\lambda_{i}\in[0,1]$ with $\sum\limits_{i=1}^{t}\lambda_{i}=1$ such that $U=\sum\limits_{i=1}^{t}\lambda_{i}U_{i}$. Combining Jensen's inequality, it concludes that
	\begin{equation*}
		\langle d,Ud\rangle=\sum_{i=1}^{t}\lambda_{i}\langle d,U_{i}d\rangle\geq\sum_{i=1}^{t}\lambda_{i}\|U_{i}d\|^{2}\geq\|Ud\|^{2},\, \forall\, d\in\mathcal{R}^{m}.
	\end{equation*}
	This completes the proof.
\end{proof}
 
Proposition \ref{prop-partial-proximal-mapping} plays a pivotal role in establishing properties related to SSOSC for problem $(P)$. By leveraging the properties of the second-order subderivative, we establish the following two key propositions. These propositions will be essential for subsequent analyses and proofs concerning SSOSC.
\begin{proposition}\label{prop-2}
	Let $r:\mathcal{R}^{m}\rightarrow (-\infty,+\infty]$ be a proper l.s.c. convex function with $\bar{x}\in\operatorname{dom}r$ and $\bar{u}\in\partial r(\bar{x})$. Suppose that the function $r$ is parabolically epi-differentiable at $\bar{x}$ for every $w\in\{d\, |\, r^{\downarrow}_{-}(\bar{x},d)=\langle \bar{u},d\rangle\}$ and parabolically regular at $\bar{x}$ for $\bar{u}$. Then we have
	\begin{equation}\label{eq:reformulation-graphical-0}
		D(\partial r)(\bar{x},\bar{u})(0)=\{d\, |\, r^{\downarrow}_{-}(\bar{x},d)=\langle \bar{u},d\rangle\}^{\circ}
	\end{equation}
	and
	$
		\{d\, |\, r^{\downarrow}_{-}(\bar{x},d)=\langle \bar{u},d\rangle\}^{\circ}=\{d\, |\, r^{*\downarrow}_{-}(\bar{u},d)=\langle \bar{x},d\rangle\}\cap\{d\, |\, d^{2}r^{*}(\bar{u},\bar{x})(d)=0\}.
	$
\end{proposition}
\begin{proof}
	By \cite[Proposition~13.20]{Rockafellar1998}, the function $d^{2}r^{*}(\bar{u},\bar{x})$ is convex with nonnegative value and 
	\begin{equation}\label{eq:graphical-derivative-0-equations}
		\begin{aligned}
		&\{d\, |\, r^{*\downarrow}_{-}(\bar{u},d)=\langle \bar{x},d\rangle\}\cap\{d\, |\, d^{2}r^{*}(\bar{u},\bar{x})(d)=0\}\\
		&=\mathop{\operatorname{argmin}}_{d\in\mathcal{R}^{m}}\{d^{2}r^{*}(\bar{u},\bar{x})(d)\, |\, r^{*\downarrow}_{-}(\bar{u},d)=\langle \bar{x},d\rangle\}=\mathop{\operatorname{argmin}}_{d\in\mathcal{R}^{m}}\{d^{2}r^{*}(\bar{u},\bar{x})(d)\}\\
		&=\{d\, |\, 0\in\partial(\frac{1}{2}d^{2}r^{*}(\bar{u},\bar{x}))(d)\}=\partial(\frac{1}{2}d^{2}r(\bar{x},\bar{u}))(0)=D(\partial r)(\bar{x},\bar{u})(0),
		\end{aligned}
	\end{equation}
	where the second equality follows from  \cite[Proposition~13.5]{Rockafellar1998}, 
	the fourth equality is due to  \cite[Theorem~13.21]{Rockafellar1998} (see also \cite[Theorem~2.4]{Rockafellar1990}), and the last equality follows from \cite[Theorem~3.8]{Mohammadi2020} and \cite[Theorem~13.40]{Rockafellar1998} (see also  \cite[Theorem~2.4]{Rockafellar1990}).
	
	Now we prove the result \eqref{eq:reformulation-graphical-0}.
	It is demonstrated by \cite[Theorem~3.8]{Mohammadi2020} that 
	$
		\operatorname{dom}d^{2}r(\bar{x},\bar{u})=\{d\, |\, r^{\downarrow}_{-}(\bar{x},d)=\langle \bar{u},d\rangle\}$.
    By the convexity of  $d^{2}r(\bar{x},\bar{u})$ along with $d^{2}r(\bar{x},\bar{u})(0)=0$ from \cite[Proposition~13.20]{Rockafellar1998},
	for any $v\in D(\partial r)(\bar{x},\bar{u})(0)$, the following inequality holds.
    \begin{equation*}
		d^{2}r(\bar{x},\bar{u})(z)-d^{2}r(\bar{x},\bar{u})(0)\geq\langle v,z-0\rangle,\quad\forall\,z\in\mathcal{R}^{m}.
    \end{equation*}
    Choose $\lambda>0$ and   $z\in\operatorname{dom}d^{2}r(\bar{x},\bar{u})$. Given that $d^{2}r(\bar{x},\bar{u})(z)\geq 0$, the aforementioned inequality is simplified to 
    \begin{equation*}
		\lambda^{2}d^{2}r(\bar{x},\bar{u})(z)\geq\lambda\langle v,z\rangle,
    \end{equation*}
	from which it deduces that $\langle v,z\rangle\leq 0$ and $z\in\{d\, |\, r^{\downarrow}_{-}(\bar{x},d)=\langle \bar{u},d\rangle\}^{\circ}$. As a result, it concludes that $D(\partial r)(\bar{x},\bar{u})(0)\subseteq\{d\, |\, r^{\downarrow}_{-}(\bar{x},d)=\langle \bar{u},d\rangle\}^{\circ}$. Conversely, for any $z\in\{d\, |\, r^{\downarrow}_{-}(\bar{x},d)=\langle \bar{u},d\rangle\}^{\circ}$ and $v\in\{d\, |\, r^{\downarrow}_{-}(\bar{x},d)=\langle \bar{u},d\rangle\}$, we have
	\begin{equation*}
		\langle v,z\rangle\leq 0\leq d^{2}r(\bar{x},\bar{u})(v)-d^{2}r(\bar{x},\bar{u})(0).
	\end{equation*}
	If $v\notin\{d\, |\, r^{\downarrow}_{-}(\bar{x},d)=\langle \bar{u},d\rangle\}$, then $d^{2}r(\bar{x},\bar{u})(v)=+\infty$, which indicates that the above inequality is valid for any $v\in\mathcal{R}^{m}$. By the convexity of $d^{2}r(\bar{x},\bar{u})$, it follows that $z\in\partial d^{2}r(\bar{x},\bar{u})(0)=D(\partial r)(\bar{x},\bar{u})(0)$. This demonstrates that $\{d\, |\, r^{\downarrow}_{-}(\bar{x},d)=\langle \bar{u},d\rangle\}^{\circ}\subseteq D(\partial r)(\bar{x},\bar{u})(0)$.
	Combining these results  with   \eqref{eq:graphical-derivative-0-equations}, the desired results follow directly.
\end{proof}

\begin{remark}
	The result presented in   equality \eqref{eq:reformulation-graphical-0} simplifies to \cite[Corollary~3.1]{MohammadBoris2020} when the function $r$ is taken as the indicator function of a convex set satisfying the specified conditions in their work.
\end{remark}

\begin{proposition}\label{prop-5}
	Let $r:\mathcal{R}^{m}\rightarrow (-\infty,+\infty]$ be a proper l.s.c. convex function with $\bar{x}\in\operatorname{dom}r$ and $\bar{u}\in\partial r(\bar{x})$. Suppose that the function $r$ is parabolically epi-differentiable at $\bar{x}$ for every $w\in\{d\, |\, r^{\downarrow}_{-}(\bar{x},d)=\langle \bar{u},d\rangle\}$ and parabolically regular at $\bar{x}$ for $\bar{u}$. Then for $d\in\mathcal{R}^{m}$,  $D\operatorname{Prox}_{r}(\bar{x}+\bar{u})(d)=\{\operatorname{Prox}_{r}'(\bar{x}+\bar{u},d)\}$. For any $h\in\mathcal{R}^{m}$, that $h= \operatorname{Prox}'_{r}(\bar{x}+\bar{u},h+d)$ holds if and only if $h$ satisfies $r^{\downarrow}_{-}(\bar{x},h)=\langle \bar{u},h\rangle$ and $d\in D(\partial r)(\bar{x},\bar{u})(h)$.
	Moreover, for any $h$ satisfying $r^{\downarrow}_{-}(\bar{x},h)=\langle \bar{u},h\rangle$
	and $d\in D(\partial r)(\bar{x},\bar{u})(h)$, we have $d^{2}r(\bar{x},\bar{u})(h)=\langle h,d\rangle$. Furthermore, we also have
	\begin{equation}\label{eq:proximal-mapping-directional-derivative}
	\begin{aligned}	\operatorname{Prox}'_{r}(\bar{x}+\bar{u},h)=h\;&\Longleftrightarrow\;\operatorname{Prox}'_{r^{*}}(\bar{x}+\bar{u},h)=0\Longleftrightarrow\; d^{2}r(\bar{x},\bar{u})(h)=0,\\
    \operatorname{Prox}'_{r^{*}}(\bar{x}+\bar{u},h)=h\;&\Longleftrightarrow\;\operatorname{Prox}'_{r}(\bar{x}+\bar{u},h)=0\Longleftrightarrow\; d^{2}r^{*}(\bar{x},\bar{u})(h)=0.
    \end{aligned}
	\end{equation}
\end{proposition}
\begin{proof}
	It follows from Proposition \ref{prop-domd2g} and \cite[Theorem~3.8]{Mohammadi2020} that
	$
		\operatorname{dom}D(\partial r)(\bar{x},\bar{u})=\operatorname{dom}d^{2}r(\bar{x},\bar{u})=\{d\, |\, r^{\downarrow}_{-}(\bar{x},d)=\langle \bar{u},d\rangle\}$. 
	Given that $\bar{u}\in\partial r(\bar{x})$ and $\operatorname{Prox}_{r}(\bar{x}+\bar{u})=\bar{x}$, by \eqref{eq:inverse-of-DS}, \cite[Theorem~3.8]{Mohammadi2020},  \cite[Theorem~13.40]{Rockafellar1998} and \cite[Proposition~2.2]{Hoheisel2012}, the proximal mapping $\operatorname{Prox}_{r}$ is directionally differentiable at $\bar{x}+\bar{u}$ with  
    \begin{equation*}
	\begin{aligned}
		\operatorname{Prox}'_{r}(\bar{x}+\bar{u},h+d)=h\;&\Longleftrightarrow\; h\in D\operatorname{Prox}_{r}(\bar{x}+\bar{u})(h+d)\\ &\Longleftrightarrow\; d\in D(\partial r)(\bar{x}, \bar{u})(h)\;\Longleftrightarrow\; d\in\partial(\frac{1}{2}d^{2}r(\bar{x},\bar{u}))(h)
		\end{aligned}
    \end{equation*}
    and $D\operatorname{Prox}_{r}(\bar{x}+\bar{u})(d)=\{\operatorname{Prox}_{r}'(\bar{x}+\bar{u},d)\}=\{\operatorname{Prox}_{\frac{1}{2}d^{2}r(\bar{x},\bar{u})}(d)\}$. 
    Consequently, if $d\in D(\partial r)(\bar{x}, \bar{u})(h)$, by \cite[Theorem~13.21]{Rockafellar1998} and \cite[Theorem~23.5 and Corollary~15.3.2]{Rockafellar1970} it yields
	\begin{equation*}
		\langle h,d\rangle=\frac{1}{2}d^{2}r(\bar{x},\bar{u})(h)+\frac{1}{2}d^{2}r^{*}(\bar{u},\bar{x})(d)\leq\sqrt{d^{2}r(\bar{x},\bar{u})(h)}\sqrt{d^{2}r^{*}(\bar{u},\bar{x})(d)},
	\end{equation*}
	which implies $\langle h,d\rangle=d^{2}r(\bar{x},\bar{u})(h)$. Combining this with Moreau's identity \eqref{eq:moreau-identity}, \cite[Proposition~13.20 and Theorems~13.21,~13.40]{Rockafellar1998}, the result \eqref{eq:proximal-mapping-directional-derivative} follows directly. This completes the proof.
\end{proof}

\section{SOSC, SSOSC and constraint qualification conditions}\label{sec:SRCQ-SOSC-SSOSC-nondegeneracy}

In this section, we introduce SOSC, SSOSC and constraint qualification conditions for the general composite optimization problem $(P)$. Let $(\bar{x},\bar{c})$ be an optimal solution of problem \eqref{general-prob-equivalent} with $(\bar{\mu},-1)$ being the corresponding Lagrange multiplier and $\mathcal{K}:=\operatorname{epi}g$. Suppose that the following Robinson constraint qualification (RCQ) for problem \eqref{general-prob-equivalent} holds at $(\bar{x},\bar{c})$, i.e.,
\begin{equation}\label{RCQ-primal-prob}
	(F'(\bar{x}),1)(\mathcal{R}^{n}\times\mathcal{R})-\mathcal{T}_{\mathcal{K}}(F(\bar{x}),\bar{c})=\mathcal{R}^{m}\times\mathcal{R},
\end{equation}
where $\mathcal{T}_{\mathcal{K}}(F(\bar{x}),\bar{c})=\{(d,d_c)\, |\, g^{\downarrow}_{-}(F(\bar{x}),d)\leq d_c\}$. It follows from \cite{IoffeandOutrata2008} that $\partial(g\circ F)(\bar{x})=F'(\bar{x})^{*}\partial g(F(\bar{x}))$. We say a point $x\in\mathcal{R}^{n}$ is a stationary point of problem $(P)$ if it satisfies $0\in F'(x)^{*}\partial g(F(x))$. 
By \cite[Theorem~3.9 and Proposition~3.17]{Bonnans2000}, the set of Lagrange multipliers is a nonempty, convex and compact set if and only if RCQ \eqref{RCQ-primal-prob} holds at $(\bar{x},\bar{c})$. RCQ for problem $(P)$ at $\bar{x}$ takes the following form
\begin{equation*}
	F'(\bar{x})\mathcal{R}^{n}-\{d\,|\,\exists\, d_c\in\mathcal{R},\, g^{\downarrow}_{-}(F(\bar{x}),d)\leq d_c\}=\mathcal{R}^{m},
\end{equation*}
which can be written equivalently as
$
	\mathcal{N}_{\operatorname{dom}g}(F(\bar{x}))\cap\operatorname{ker}F'(\bar{x})^{*}=\{0\}$ or
	$\partial^{\infty}g(F(\bar{x}))\cap\operatorname{ker}F'(\bar{x})^{*}=\{0\}$,
by \cite[Proposition~8.12 and Exercise~8.23]{Rockafellar1998}. The strict Robinson constraint qualification (SRCQ) for problem \eqref{general-prob-equivalent} is said to hold at $(\bar{x},\bar{c})$ with $(\bar{\mu},-1)$ if
\begin{equation*}\label{SRCQ-primal-prob}
	(F'(\bar{x}),1)(\mathcal{R}^{n}\times\mathcal{R})-\mathcal{T}_{\mathcal{K}}(F(\bar{x}),\bar{c})\cap(\bar{\mu},-1)^{\perp}=\mathcal{R}^{m}\times\mathcal{R}.
\end{equation*}
By \cite[Corollary~2.4.9]{Clarke1990} and \cite[Theorem~8.9]{Rockafellar1998}, $g^{\downarrow}_{-}(F(\bar{x}),d)\geq \langle\bar{\mu},d\rangle$ holds for any  $d\in\mathcal{R}^{m}$ and  $(\bar{\mu},-1)\in\mathcal{N}_{\mathcal{K}}(F(\bar{x}),g(F(\bar{x})))$ is equivalent to $\bar{\mu}\in\partial g(F(\bar{x}))$.
Then
$\mathcal{T}_{\mathcal{K}}(F(\bar{x}),\bar{c})\cap(\bar{\mu},-1)^{\perp}=\{(d,d_c)\, |\, g^{\downarrow}_{-}(F(\bar{x}),d)=\langle\bar{\mu},d\rangle=d_c\}$ and
SRCQ for problem $(P)$ can be written as
\begin{equation}\label{eq:srcq}
	F'(\bar{x})\mathcal{R}^{n}-\{d\, |\, g^{\downarrow}_{-}(F(\bar{x}),d)=\langle\bar{\mu},d\rangle\}=\mathcal{R}^{m}.
\end{equation}
The corresponding Lagrange multiplier set is a singleton if SRCQ \eqref{eq:srcq} holds at $\bar{x}$ as demonstrated in \cite[Proposition~4.50]{Bonnans2000}.

Suppose that $g$ is $\mathcal{C}^{2}$-cone reducible at $\bar{x}$. The nondegeneracy condition for problem (\ref{general-prob-equivalent}) holds at $(\bar{x},\bar{c})$ if
\begin{equation*}
	(F'(\bar{x}),1)(\mathcal{R}^{n}\times\mathcal{R})-\operatorname{lin}\mathcal{T}_{\mathcal{K}}(F(\bar{x}),\bar{c})=\mathcal{R}^{m}\times\mathcal{R}.
\end{equation*}
While the nondegeneracy condition for problem $(P)$ takes the following form
\begin{equation*}
F'(\bar{x})\mathcal{R}^{n}-\{d\, |\, g^{\downarrow}_{-}(F(\bar{x}),d)\leq-g^{\downarrow}_{-}(F(\bar{x}),-d)\}=\mathcal{R}^{m},
\end{equation*}
which can be equivalently reformulated as
\begin{equation}\label{eq:comp-prob-nondegeneracy-cond}
	F'(\bar{x})\mathcal{R}^{n}-\{d\, |\, g^{\downarrow}_{-}(F(\bar{x}),d)=-g^{\downarrow}_{-}(F(\bar{x}),-d)\}=\mathcal{R}^{m},
\end{equation}
by \cite[Corollary~4.7.2]{Rockafellar1970}.
Define the critical cone of problem (\ref{general-prob-equivalent}) at $(\bar{x},\bar{c})$ as
\begin{equation*}\label{eq:critical-cone-problem(1)}
	\begin{aligned}
	\widehat{\mathcal{C}}(\bar{x},\bar{c})&=\{(d,d_c)\, |\, g^{\downarrow}_{-}(F(\bar{x}),F'(\bar{x})d)=0,d_c=0\}\\
	&=\{(d,d_c)\, |\, g^{\downarrow}_{-}(F(\bar{x}),F'(\bar{x})d)=\langle F'(\bar{x})d,\bar{\mu}\rangle=0,d_c=0\}
	\end{aligned}
\end{equation*}
and the critical cone of problem $(P)$ at $\bar{x}$ as
\begin{equation*}\label{eq:critical-cone-problem(P)}
	\mathcal{C}(\bar{x})=\{d\, |\, g^{\downarrow}_{-}(F(\bar{x}),F'(\bar{x})d)=0\}=\{d\, |\, F'(\bar{x})d\in\mathcal{N}_{\partial g(F(\bar{x}))}(\bar{\mu})\}.
\end{equation*}
We say SOSC for problem \eqref{general-prob-equivalent} holds at $(\bar{x},\bar{c})$ if
\begin{equation*}
	\sup_{\mu\in\Lambda(\bar{x})}\{\langle\tilde{d},\nabla^{2}\mathcal{L}((\cdot);(\mu,-1))\tilde{d}\rangle(\bar{x},\bar{c})-\sup_{(w,\gamma)\in\mathcal{T}^{2}}\{\langle\mu,w\rangle-\gamma\}\}> 0,\, \forall\, \tilde{d}\in\widehat{\mathcal{C}}(\bar{x},\bar{c})\setminus\{0\},
\end{equation*}
where $\tilde{d}=(d,d_{c})$, $\Lambda(\bar{x})$ denotes the set of all Lagrange multipliers associated with $\bar{x}$ for problem $(P)$, and
$
\mathcal{T}^{2}:=\mathcal{T}^{2}_{\operatorname{epi}g}((F(\bar{x}),\bar{c}),(F'(\bar{x})d,g^{\downarrow}_{-}(F(\bar{x}),F'(\bar{x})d)))$. 
For simplicity, define
$
\varphi_{g}(
x,\mu)(d):=\inf\limits_{w\in\mathcal{R}^{m}}\{g^{\downdownarrows}_{-}(x;d,w)-\langle w,\mu\rangle\}
$.
It follows from  \cite[Proposition~3.30]{Bonnans2000} that $-\sup\limits_{(w,\gamma)\in\mathcal{T}^{2}}\{\langle\mu,w\rangle-\gamma\}=\varphi_{g}(
F(\bar{x}),\mu)(F'(\bar{x})d)$. Then 
SOSC at $(\bar{x},\bar{c})$ for problem \eqref{general-prob-equivalent} can be equivalently formulated as follows:
\begin{equation}
    \label{eq:SOSC}
    \sup_{\mu\in\Lambda(\bar{x})}\{\langle\mu,F''(\bar{x})(d,d)\rangle+\varphi_{g}(F(\bar{x}),\mu)(F'(\bar{x})d)\}> 0,\, \forall\, d\in\mathcal{C}(\bar{x})\setminus\{0\},
\end{equation}
which is also referred to as SOSC at $\bar{x}$ for problem $(P)$. Consequently, SOSC at $(\bar{x},\bar{c})$ for problem \eqref{general-prob-equivalent} coincides with SOSC at $\bar{x}$ for problem $(P)$.

Let $r:\mathcal{R}^{m}\rightarrow
(-\infty,+\infty]$ be an l.s.c. function. Given $\bar{x}\in\operatorname{dom}r$ and $\bar{u}\in\partial r(\bar{x})$, define a second-order variational function   $\Gamma_{r}(\bar{x},\bar{u}):\mathcal{R}^{m}\rightarrow\overline{\mathcal{R}}$ as 
\begin{equation*}
	\Gamma_{r}(\bar{x},\bar{u})(v):=\left\{\begin{array}{cl}\min\limits_{{d\in\mathcal{R}^{m},v=Ud}\atop{U\in\partial \operatorname{Prox}_{r}(\bar{x}+\bar{u})}}\langle v,d-v\rangle,&\mbox{if}\ v\in\mathop{\cup}\limits_{U\in\partial \operatorname{Prox}_{r}(\bar{x}+\bar{u})}\operatorname{rge}U,\\
		+\infty,&\mbox{otherwise}.
	\end{array}\right.
\end{equation*}
Although the definition of the second-order variational function  $\Gamma_{r}(\bar{x},\bar{u})$ appears complex, it becomes computationally tractable when the Clarke subdifferential of the proximal mapping $\operatorname{Prox}_{r}$ at $\bar{x}+\bar{u}$ admits a closed-form expression, then the second-order variational function  $\Gamma_{r}(\bar{x},\bar{u})$ can  be evaluated explicitly. Furthermore,
it follows from Proposition \ref{prop-partial-proximal-mapping} that $\Gamma_{r}(\bar{x},\bar{u})(v)\geq 0$ for any $v\in\operatorname{dom}\Gamma_{r}(\bar{x},\bar{u})$, provided that $r$ is a proper l.s.c. convex function. Now we formally present the definition of SSOSC for problem $(P)$.
\begin{definition}
    Let $\bar{x}$ be a stationary point of problem $(P)$ with $\Lambda(\bar{x})$ as the corresponding Lagrange multiplier set. We say that SSOSC holds at $\bar{x}$ if
	\begin{equation}
		\sup_{\mu\in\Lambda(\bar{x})}\{\langle\mu,F''(\bar{x})(d,d)\rangle+\Gamma_{g}(F(\bar{x}),\mu)(F'(\bar{x})d)\}> 0,\,\forall\,d\neq0.\label{eq:comp-prob-SSOSC}
	\end{equation}
\end{definition}
The most intricate component in SSOSC \eqref{eq:comp-prob-SSOSC} is the second term involving the second-order variational function $\Gamma_{g}(F(\bar{x}),\mu)$. For a given proper function $r$, we first establish foundational properties of  $\Gamma_{r}$ and the proximal mapping $\operatorname{Prox}_{r}$ at the point under consideration to enhance the understanding of SSOSC.
\begin{lemma}\label{lemma-Gamma-inequality}
	Suppose that the function $r:\mathcal{R}^{m}\rightarrow(-\infty,+\infty]$ is  proper, l.s.c. and convex. Given $\bar{x}\in\operatorname{dom}r$ and $\bar{u}\in\partial r(\bar{x})$, for any $U\in\partial\operatorname{Prox}_{r}(\bar{x}+\bar{u})$ and $v,d\in\mathcal{R}^{m}$ which satisfy $v=U(v+d)$, we have
	\begin{equation}		
		\Gamma_{r}(\bar{x},\bar{u})(v)\leq\langle v,d\rangle.\label{eq:corollary-result-2}
	\end{equation}
	Furthermore, one also has
	\begin{equation}\label{eq:prop8-result-3}
		\mathop{\cup}_{U\in\partial \operatorname{Prox}_{r}(\bar{x}+\bar{u})}\operatorname{ker}U=\{v\, |\, \Gamma_{r^{*}}(\bar{u},\bar{x})(v)=0\}.
	\end{equation}
\end{lemma}
\begin{proof}
	The result \eqref{eq:corollary-result-2} follows directly from the definition of $\Gamma_{r}(\bar{x},\bar{u})$. We proceed to prove the second result \eqref{eq:prop8-result-3}. For any $v'\in\mathop{\cup}\limits_{U\in\partial \operatorname{Prox}_{r}(\bar{x}+\bar{u})}\operatorname{ker}U$, there exists $U\in\partial\operatorname{Prox}_{r}(\bar{x}+\bar{u})$ such that $Uv'=0$, which is also equivalent to $v'=(I-U)v'$ with $I-U\in\partial\operatorname{Prox}_{r^{*}}(\bar{x}+\bar{u})$. Applying the result \eqref{eq:corollary-result-2} to $\Gamma_{r^{*}}(\bar{u},\bar{x})$ with $v=v'$ and $d=0$, then  $\Gamma_{r^{*}}(\bar{u},\bar{x})(v')\leq0$. By \cite[Theorem~12.2]{Rockafellar1970} and Proposition \ref{prop-partial-proximal-mapping}, the function $r^{*}$ is a proper l.s.c. convex function and $\Gamma_{r^{*}}(\bar{u},\bar{x})(v)\geq0$ for all $v\in\mathcal{R}^{m}$. Therefore, it follows that $\Gamma_{r^{*}}(\bar{u},\bar{x})(v')=0$. Conversely, if $\Gamma_{r^{*}}(\bar{u},\bar{x})(v)=0$, there exist $V\in\partial\operatorname{Prox}_{r^{*}}(\bar{x}+\bar{u})$ and $d\in\mathcal{R}^{m}$ with $v=Vd$ such that $\Gamma_{r^{*}}(\bar{u},\bar{x})(v)=\langle v,d-v\rangle=0$. Then $\langle Vd,(I-V)d\rangle=\langle d,(V-V^{2})d\rangle=0$. By Proposition \ref{prop-partial-proximal-mapping}, the matrix $V-V^{2}$ is positive semidefinite and then $(V-V^{2})d=0$ which also implies that $(I-V)v=0$. Consequently this confirms the existence of  $U=(I-V)\in\partial\operatorname{Prox}_{r}(\bar{x}+\bar{u})$ with $v\in\operatorname{ker}U$ and completes the proof of \eqref{eq:prop8-result-3}.
\end{proof}

\begin{lemma}\label{lemma-range-partial-derivative}
	Let   $r:\mathcal{R}^{m}\rightarrow(-\infty,+\infty]$ be a proper l.s.c. convex function with $\bar{x}\in\operatorname{dom}r$ and $\bar{u}\in\partial r(\bar{x})$. We have
	\begin{equation*}
		\operatorname{dom}D^{*}(\partial r)(\bar{x},\bar{u})\subseteq\operatorname{conv}\mathop{\cup}_{U\in\partial\operatorname{Prox}_{r}(\bar{x}+\bar{u})}\operatorname{rge}U\subseteq\operatorname{aff}(\operatorname{rge}(D\operatorname{Prox}_{r}(\bar{x}+\bar{u}))).
 	\end{equation*}
    \end{lemma}	
\begin{proof}
	We first show the following equality
	\begin{equation}\label{eq:critical-cone-polar-equivalent}
		(\operatorname{rge}(D\operatorname{Prox}_{r}(\bar{x}+\bar{u})))^{\circ}=\{d\, |\, 0\in\widehat{D}^{*}\operatorname{Prox}_{r}(\bar{x}+\bar{u})(-d)\}.
	\end{equation}
	Let $y$ be an arbitrary element in the left-hand side set of \eqref{eq:critical-cone-polar-equivalent}. For any $(w,z)\in\mathcal{T}_{\operatorname{gph} \operatorname{Prox}_{r}}(\bar{x}+\bar{u},\bar{x})$, particularly with $z\in\operatorname{rge}(D\operatorname{Prox}_{r}(\bar{x}+\bar{u}))$, it holds that  $\langle(0,y),(w,z)\rangle=\langle y,z\rangle\leq 0$.
	This implies that $(0,y)\in\widehat{\mathcal{N}}_{\operatorname{gph} \operatorname{Prox}_{r}}(\bar{x}+\bar{u},\bar{x})$ and consequently 
	$0\in\widehat{D}^{*}\operatorname{Prox}_{r}(\bar{x}+\bar{u})(-y)$. Conversely, suppose $\tilde{y}$ satisfies $0\in\widehat{D}^{*}\operatorname{Prox}_{r}(\bar{x}+\bar{u})(-\tilde{y})$. Given any $z\in\operatorname{rge}(D\operatorname{Prox}_{r}(\bar{x}+\bar{u}))$, there exists $w\in\mathcal{R}^{m}$ such that  $(w,z)\in\mathcal{T}_{\operatorname{gph}\operatorname{Prox}_{r}}(\bar{x}+\bar{u},\bar{x})$. Then $
	\langle\tilde{y},z\rangle=\langle(0,\tilde{y}),(w,z)\rangle\leq 0$,
	which confirms that $\tilde{y}\in(\operatorname{rge}(D\operatorname{Prox}_{r}(\bar{x}+\bar{u})))^{\circ}$, as required.
	
	To establish the desired result, the equality  \eqref{eq:critical-cone-polar-equivalent} tells us that
	\begin{equation}
	\label{eq:polar-aff-rge-graphical-proximal-mapping}	(\operatorname{aff}(\operatorname{rge}(D\operatorname{Prox}_{r}(\bar{x}+\bar{u}))))^{\circ}=\operatorname{lin}\{d\, |\, 0\in\widehat{D}^{*}\operatorname{Prox}_{r}(\bar{x}+\bar{u})(-d)\}.
	\end{equation}
	For any $d\in\operatorname{lin}\{d\, |\, 0\in\widehat{D}^{*}\operatorname{Prox}_{r}(\bar{x}+\bar{u})(-d)\}$, observing the definition of regular normal cone $\widehat{\mathcal{N}}_{\operatorname{gph} \operatorname{Prox}_{r}}(\bar{x}+\bar{u},\bar{x})$ and
	\begin{equation*}
		v\in\widehat{D}^{*}\operatorname{Prox}_{r}(\bar{x}+\bar{u})(y)\, \Longleftrightarrow\, \langle v,x-\bar{x}-\bar{u}\rangle\leq\langle y,u-\bar{x}\rangle+o(\|(x,u)-(\bar{x}+\bar{u},\bar{x})\|),\\ \forall\, u=\operatorname{Prox}_{r}(x),
	\end{equation*}
	 it follows that
	\begin{equation}\label{eq:prox-result}
		\langle d,\operatorname{Prox}_{r}(x)-\operatorname{Prox}_{r}(\bar{x}+\bar{u})\rangle=o(\|(x,u)-(\bar{x}+\bar{u},\bar{x})\|)=o(\|x-\bar{x}-\bar{u}\|).
	\end{equation}
	Let $\mathcal{D}_{F}$ denote the set of points where $\operatorname{Prox}_{r}$ is differentiable. For any sequence $\{x^{k}\}\subseteq\mathcal{D}_{F}$ with $x^{k}\rightarrow\bar{x}+\bar{u}$, the expansion $\operatorname{Prox}_{r}(x^{k})-\operatorname{Prox}_{r}(\bar{x}+\bar{u})=\operatorname{Prox}'_{r}(x^{k})(x^{k}-\bar{x}-\bar{u})+o(\|x^{k}-\bar{x}-\bar{u}\|)$ is valid for sufficiently large $k$. Substituting this into \eqref{eq:prox-result} yields
	$
	\langle d,\operatorname{Prox}'_{r}(x^{k})(x^{k}-\bar{x}-\bar{u})\rangle=o(\|x^{k}-\bar{x}-\bar{u}\|)$.
	Passing to the limit as $k\rightarrow+\infty$, it follows that  $Ud=0$ for any $U\in\partial_{B}\operatorname{Prox}_{r}(\bar{x}+\bar{u})$. Then taking into account \eqref{eq:polar-aff-rge-graphical-proximal-mapping} we obtain
	\begin{equation*}
		(\operatorname{aff}(\operatorname{rge}(D\operatorname{Prox}_{r}(\bar{x}+\bar{u}))))^{\circ}\subseteq\mathop{\cap}_{U\in\partial_{B} \operatorname{Prox}_{r}(\bar{x}+\bar{u})}\operatorname{ker}U.
	\end{equation*}
    By applying the polar operation to both sides of the above inclusion, it reduces to
	\begin{equation*}
		\begin{aligned}
		\operatorname{aff}(\operatorname{rge}(D\operatorname{Prox}_{r}(\bar{x}+\bar{u})))&\supseteq\sum_{U\in\partial_{B}\operatorname{Prox}_{r}(\bar{x}+\bar{u})}\operatorname{rge}U^{*}=\sum_{U\in\partial_{B}\operatorname{Prox}_{r}(\bar{x}+\bar{u})}\operatorname{rge}U\\
		&=\operatorname{conv}\{\operatorname{rge}(U)\,|\, U\in\partial_{B} \operatorname{Prox}_{r}(\bar{x}+\bar{u})\}\\
		&=\operatorname{conv}\mathop{\cup}_{U\in\partial_{B} \operatorname{Prox}_{r}(\bar{x}+\bar{u})}\operatorname{rge}U=\operatorname{conv}\mathop{\cup}_{U\in\partial \operatorname{Prox}_{r}(\bar{x}+\bar{u})}\operatorname{rge}U,
		\end{aligned}
	\end{equation*}
	 where the second equality follows from  \cite[Theorem~3.8]{Rockafellar1970}.   The desired result follows immediately from the relationship
	 $\operatorname{dom}D^{*}(\partial r)(\bar{x},\bar{u})=-\operatorname{rge}D^{*}\operatorname{Prox}_{r}(\bar{x}+\bar{u})$ and \cite[Theorem~13.52]{Rockafellar1998}.
\end{proof}

\begin{remark}
	Verifying the inverse inclusion in Lemma \ref{lemma-range-partial-derivative} remains challenging. If valid, Proposition \ref{prop-domd2g} shows the first equality of Assumption \ref{assumption-2} would be unnecessary. For practical verification, we provide an easily verifiable sufficient condition for Assumption \ref{assumption-2} in the appendix. 
\end{remark}

\begin{corollary}\label{corollary-2}
	Let  $r:\mathcal{R}^{m}\rightarrow(-\infty,+\infty]$ be a proper l.s.c. convex function with $\bar{x}\in\operatorname{dom}r$ and $\bar{u}\in\partial r(\bar{x})$. Then we have
	\begin{equation}
	\label{eq:null-coderivative-proximal-mapping}	\begin{aligned}
		\{d\, |\, 0\in D^{*}\operatorname{Prox}_{r}(\bar{x}+\bar{u})(d)\}&=\mathop{\cup}_{W\in\partial\operatorname{Prox}_{r}(\bar{x}+\bar{u})}\operatorname{ker}W\\
		&=\operatorname{aff}\{d\, |\, r^{*\downarrow}_{-}(\bar{u},d)=\langle\bar{x},d\rangle\}\cap(\mathop{\cup}_{W\in\partial\operatorname{Prox}_{r}(\bar{x}+\bar{u})}\operatorname{ker}W).
		\end{aligned}
	\end{equation}
\end{corollary}
\begin{proof}
	We first establish the second equality in \eqref{eq:null-coderivative-proximal-mapping}. For any
	$
	d\in\mathop{\cup}\limits_{W\in\partial \operatorname{Prox}_{r}(\bar{x}+\bar{u})}\operatorname{ker}W$, there exists $W\in\partial\operatorname{Prox}_{r}(\bar{x}+\bar{u})$ such that $Wd=0$, which is also equivalent to $d=(I-W)d$ with $I-W\in\partial\operatorname{Prox}_{r^{*}}(\bar{x}+\bar{u})$. Then the inclusion
	\begin{equation}
	\label{inclusion-ker-partial-proximal-mapping}	\mathop{\cup}_{W\in\partial\operatorname{Prox}_{r}(\bar{x}+\bar{u})}\operatorname{ker}W\subseteq\mathop{\cup}_{U\in\partial\operatorname{Prox}_{r^{*}}(\bar{x}+\bar{u})}\operatorname{rge}U
	\end{equation}
	holds.
	It follows from \cite[Proposition~13.5]{Rockafellar1998}, \cite[Theorem~12.2]{Rockafellar1970} and \cite[Proposition~2.126]{Bonnans2000} that
	\begin{equation*}
		\begin{aligned}
	\operatorname{rge}(D\operatorname{Prox}_{r^{*}}(\bar{x}+\bar{u}))&=\operatorname{dom}D(\partial r^{*})(\bar{u},\bar{x})\subseteq\{d\,|\,r^{*\downarrow}_{-}(\bar{u},d)=\langle\bar{x},d\rangle\}=\mathcal{N}_{\partial r^{*}(\bar{u})}(\bar{x})
	\end{aligned}
	\end{equation*}
	and the set $\mathcal{N}_{\partial r^{*}(\bar{u})}(\bar{x})$ is convex. Combining this with \eqref{inclusion-ker-partial-proximal-mapping} and applying Lemma \ref{lemma-range-partial-derivative} to $r^{*}$ at $\bar{u}$ with $\bar{x}$, it obtains the following inclusion relationship
	\begin{equation*}
		\mathop{\cup}_{W\in\partial\operatorname{Prox}_{r}(\bar{x}+\bar{u})}\operatorname{ker}W\subseteq\operatorname{aff}\{d\, |\, r^{*\downarrow}_{-}(\bar{u},d)=\langle\bar{x},d\rangle\},
	\end{equation*}
	as required.
    
	Now we show the first equality in \eqref{eq:null-coderivative-proximal-mapping}. For any  $d\in\mathop{\cup}\limits_{W\in\partial\operatorname{Prox}_{r}(\bar{x}+\bar{u})}\operatorname{ker}W$, by \cite[Theorem~13.52]{Rockafellar1998} and Proposition \ref{prop-partial-proximal-mapping}, it yields that $0\in \operatorname{conv}D^{*}\operatorname{Prox}_{r}(\bar{x}+\bar{u})(d)$ and $\langle d,z\rangle\geq0$ for any $z\in\operatorname{conv}D^{*}\operatorname{Prox}_{r}(\bar{x}+\bar{u})(d)$. Consequently,  \begin{equation}\label{eq:0=max_dWd}
		\begin{aligned}
			0&=\max\{\langle -d,\operatorname{conv}D^{*}\operatorname{Prox}_{r}(\bar{x}+\bar{u})(d)\rangle\}\\
			&=\max\{\sum_{i}\lambda_{i}\langle-d,W_{i}d\rangle, W_{i}\in\partial_{B}\operatorname{Prox}_{r}(\bar{x}+\bar{u}), \sum_{i}\lambda_{i}=1, \lambda_{i}\geq0\}\\
			&=\max\{\langle-d,Wd\rangle, W\in\partial_{B}\operatorname{Prox}_{r}(\bar{x}+\bar{u})\},
		\end{aligned}
	\end{equation}
	where the second equality in  \eqref{eq:0=max_dWd} is based on Carath\'{e}odory's theorem (see e.g., \cite[Theorem~17.1]{Rockafellar1970}).
	Then there exists $U\in\partial_{B}\operatorname{Prox}_{r}(\bar{x}+\bar{u})$ such that $Ud=0$. Taking into account \cite[Proposition~2.4]{Gfrerer2022} and observing $0=Ud\in D^{*}\operatorname{Prox}_{r}(\bar{x}+\bar{u})(d)$, we obtain the inclusion relationship
	\begin{equation*}
		\mathop{\cup}_{W\in\partial\operatorname{Prox}_{r}(\bar{x}+\bar{u})}\operatorname{ker}W\subseteq\{d\, |\, 0\in D^{*}\operatorname{Prox}_{r}(\bar{x}+\bar{u})(d)\}.
	\end{equation*}
	The converse inclusion is an immediate consequence of \cite[Theorem~13.52]{Rockafellar1998} and this completes the proof.
\end{proof}

Given $\bar{x}\in \operatorname{dom}r$ and $\bar{u}\in\partial r(\bar{x})$, Lemma \ref{lemma-range-partial-derivative} and  Corollary \ref{corollary-2} establish the relationship between $\partial \operatorname{Prox}_{r}(\bar{x}+\bar{u})$ and $D\operatorname{Prox}_{r}(\bar{x}+\bar{u})$, as well as $D^{*}\operatorname{Prox}_{r}(\bar{x}+\bar{u})$, respectively. Based on these results, we further derive additional results that are essential for proving the main theorems of this paper.
\begin{lemma}\label{lemma-relation-range-partial} 
	Let $r:\mathcal{R}^{m}\rightarrow (-\infty,+\infty]$ be a proper l.s.c. convex function with $\bar{x}\in\operatorname{dom}r$ and $\bar{u}\in\partial r(\bar{x})$. Suppose that the function $r$ is parabolically epi-differentiable at $\bar{x}$ for every $w\in\{d\, |\, r^{\downarrow}_{-}(\bar{x},d)=\langle \bar{u},d\rangle\}$, parabolically regular at $\bar{x}$ for $\bar{u}$, and satisfies Assumption \ref{assumption-2} at $\bar{x}$ for $\bar{u}$. One has
	\begin{equation*}
		\operatorname{conv}\mathop{\cup}_{U\in\partial \operatorname{Prox}_{r}(\bar{x}+\bar{u})}\operatorname{rge}U=
		\mathop{\cup}_{U\in\partial \operatorname{Prox}_{r}(\bar{x}+\bar{u})}\operatorname{rge}U
		=\operatorname{aff}\{d\, |\, r^{\downarrow}_{-}(\bar{x},d)=\langle\bar{u},d\rangle\}.
	\end{equation*}
\end{lemma}
\begin{proof}
	It follows from Proposition \ref{prop-domd2g}, \cite[Theorem~3.8]{Mohammadi2020} and \eqref{eq:inverse-of-DS} that
	\begin{equation*}
		\operatorname{aff}(\operatorname{rge}(D\operatorname{Prox}_{r}(\bar{x}+\bar{u})))=\operatorname{aff}(\operatorname{dom}D(\partial r)(\bar{x},\bar{u}))=\operatorname{aff}\{d\, |\, r^{\downarrow}_{-}(\bar{x},d)=\langle\bar{u},d\rangle\}.
	\end{equation*}
	Applying Assumption \ref{assumption-2},  \cite[Theorem~13.52]{Rockafellar1998}, and noticing that
	\begin{equation*}
		\begin{aligned}
		\operatorname{dom}D^{*}\partial r(\bar{x},\bar{u})&=-\operatorname{rge}(D^{*}\operatorname{Prox}_{r}(\bar{x}+\bar{u}))\subseteq\mathop{\cup}_{U\in\partial \operatorname{Prox}_{r}(\bar{x}+\bar{u})}\operatorname{rge}U,
		\end{aligned}
	\end{equation*}
	we have
	\begin{equation*}
		\operatorname{aff}\{d\, |\, r^{\downarrow}_{-}(\bar{x},d)=\langle\bar{u},d\rangle\}\subseteq\mathop{\cup}_{U\in\partial \operatorname{Prox}_{r}(\bar{x}+\bar{u})}\operatorname{rge}U\subseteq\operatorname{conv}(\mathop{\cup}_{U\in\partial \operatorname{Prox}_{r}(\bar{x}+\bar{u})}\operatorname{rge}U).
	\end{equation*}
	Together with Lemma \ref{lemma-range-partial-derivative}, this  completes the proof.
\end{proof}

\begin{proposition}\label{prop-4} 
	Let $r:\mathcal{R}^{m}\rightarrow (-\infty,+\infty]$ be a proper l.s.c. convex function with  $\bar{x}\in\operatorname{dom}r$ and $\bar{u}\in\partial r(\bar{x})$. Suppose that the function $r$ is parabolically epi-differentiable at $\bar{x}$ for every $w\in\{d\, |\, r^{\downarrow}_{-}(\bar{x},d)=\langle \bar{u},d\rangle\}$, parabolically regular at $\bar{x}$ for $\bar{u}$, and satisfies  Assumption \ref{assumption-2} at $\bar{x}$ for $\bar{u}$. Then we have
	\begin{equation*}
		\{d\, |\, r^{\downarrow}_{-}(\bar{x},d)=-r^{\downarrow}_{-}(\bar{x},-d)\}^{\circ}=\operatorname{aff}\{d\, |\, r^{*\downarrow}_{-}(\bar{u},d)=\langle \bar{x},d\rangle\}\cap\{d\, |\, \Gamma_{r^{*}}(\bar{u},\bar{x})(d)=0\}.
	\end{equation*}
\end{proposition}
\begin{proof}
    It follows directly from \cite[Theorem~8.30]{Rockafellar1998}  that 
	\begin{equation*}
		\begin{aligned}
		\{d\, |\, r^{\downarrow}_{-}(\bar{x},d)= -r^{\downarrow}_{-}(\bar{x},-d)\}&=\{d\, |\, r^{\downarrow}_{-}(\bar{x},d)=-r^{\downarrow}_{-}(\bar{x},-d)=\langle\bar{u},d\rangle\}\\
		&=\{d\, |\, r^{\downarrow}_{-}(\bar{x},d)=\langle \bar{u},d\rangle\}\cap\{d\,|\, r^{\downarrow}_{-}(\bar{x},-d)=\langle \bar{u},-d\rangle\}.
		\end{aligned}
	\end{equation*}
	Taking the polar operation and invoking Propositions \ref{prop-2}, \ref{prop-5}, it yields that
	\begin{equation*}
		\begin{aligned}
		\{d\, |\, r^{\downarrow}_{-}(\bar{x},d)= -r^{\downarrow}_{-}(\bar{x},-d)\}^{\circ}
		&=\{d\, |\, r^{\downarrow}_{-}(\bar{x},d)=\langle \bar{u},d\rangle\}^{\circ}+\{d\, |\, r^{\downarrow}_{-}(\bar{x},-d)=\langle \bar{u},-d\rangle\}^{\circ}\\
		&=\{d\, |\, r^{*\downarrow}_{-}(\bar{u},d)=\langle \bar{x},d\rangle\}\cap\{d\,|\, d^{2}r^{*}(\bar{u},\bar{x})(d)=0\}\qquad\\
		&\quad+\{d\, |\, r^{*\downarrow}_{-}(\bar{u},-d)=\langle \bar{x},-d\rangle\}\cap\{d\, |\, d^{2}r^{*}(\bar{u},\bar{x})(-d)=0\}\\
		&=\operatorname{aff}\{d\, |\, 0\in D\operatorname{Prox}_{r}(\bar{x}+\bar{u})(d)\}.
		\end{aligned}
	\end{equation*}
	Finally, combining this result with Assumption \ref{assumption-2}, Lemma \ref{lemma-Gamma-inequality} and Corollary \ref{corollary-2} completes the proof.
\end{proof}
\begin{remark}
	The second equality in Assumption \ref{assumption-2} is essential for establishing the result in Proposition \ref{prop-4}. Although its general validity for $\mathcal{C}^{2}$-cone reducible convex functions is still open, we establish a checkable sufficient condition in Appendix \ref{Appendix-A} and verify it for the indicator function of the positive semidefinite cone in Appendix \ref{Appendix-B}. 
\end{remark}

Now we are ready to introduce some equivalent formulations of SSOSC for problem $(P)$. Let $g$ be a proper l.s.c. convex function with $F(\bar{x})\in\operatorname{dom}g$ and $\bar{\mu}\in\partial g(F(\bar{x}))$. Assume that $g$ is parabolically epi-differentiable at $F(\bar{x})$ for every $w\in\{d\, |\, g^{\downarrow}_{-}(F(\bar{x}),d)=\langle \bar{\mu},d\rangle\}$, parabolically regular at $F(\bar{x})$ for $\bar{\mu}$, and satisfies Assumption \ref{assumption-2} at $F(\bar{x})$ for $\bar{\mu}$.
By Lemma \ref{lemma-relation-range-partial}, the second-order variational function $\Gamma_{g}(F(\bar{x}),\bar{\mu})$ admits the equivalent representation
\begin{equation*}
	\Gamma_{g}(F(\bar{x}),\bar{\mu})(v)=\left\{\begin{array}{ll}\min\limits_{{d\in\mathcal{R}^{m},v=Ud}\atop{U\in\partial \operatorname{Prox}_{g}(F(\bar{x})+\bar{\mu})}}\langle v,d-v\rangle,\ \mbox{if}\ v\in\operatorname{aff}\{d\, |\, g^{\downarrow}_{-}(F(\bar{x}),d)=\langle\bar{\mu},d\rangle\},\\
		\qquad\qquad\qquad+\infty,\qquad\quad\mbox{otherwise}.
	\end{array}\right.
\end{equation*}
Alternatively, it can be expressed as
\begin{equation}\label{eq:equiv-Gamma-g}
	\Gamma_{g}(F(\bar{x}),\bar{\mu})(v)=\left\{\begin{array}{cl}\min\limits_{{d\in\mathcal{R}^{m},v=Ud}\atop{U\in\partial \operatorname{Prox}_{g}(F(\bar{x})+\bar{\mu})}}\langle v,d-v\rangle,&\mbox{if}\ v\in\operatorname{rge}(D^{*}\operatorname{Prox}_{g}(F(\bar{x})+\bar{\mu})),\\
		+\infty,&\mbox{otherwise}.
	\end{array}\right.
\end{equation}

\section{The nonsingularity of Clarke's generalized Jacobian of  $R$}\label{sec:equiv-nonsignularity-SSOSC}
In this section, we explore equivalent characterizations for the nonsingularity of Clarke's generalized Jacobian of the mapping $R$ defined in \eqref{eq:kkt-generalized-equation} at a solution point.
Recall the fundamental definition of strong regularity.
\begin{definition}
	Let $\bar{x}$ be a solution to the generalized equation \eqref{eq:generalized-equation}. It is said that $\bar{x}$ is a strongly regular solution of the generalized equation \eqref{eq:generalized-equation} if there exist neighborhoods $\mathcal{U}$ of $\bar{x}$ and $\mathcal{V}$ of the origin such that for every $\delta\in\mathcal{V}$, the linearized generalized equation
	\begin{equation*}
	\delta\in\phi(\bar{x})+\phi'(\bar{x})(x-\bar{x})+\mathcal{M}(x)
	\end{equation*}
	admits a unique solution $\xi(\delta)$ in $\mathcal{U}$, and the mapping $\xi:\mathcal{V}\rightarrow\mathcal{U}$ is Lipschitz continuous.
\end{definition}	

Recall that a vector-valued function $\Phi:\mathcal{O}\subseteq\mathcal{R}^{t}\rightarrow\mathcal{R}^{t}$ is said to be a locally Lipschitz homeomorphism at $x\in\mathcal{O}$, if there exists an open neighborhood $\mathcal{N}\subseteq\mathcal{O}$ around $x$ such that the restricted mapping $\Phi|_{\mathcal{N}}$ is bijective, and both $\Phi|_{\mathcal{N}}$ and $(\Phi|_{\mathcal{N}})^{-1}$ are Lipschitz continuous on their respective domains.

Now we establish the connection between strong regularity of the KKT system \eqref{eq:comp-prob-kkt} and the locally Lipschitz homeomorphism property of the mapping $R$ around a solution point of the system  \eqref{eq:comp-prob-kkt}.

\begin{lemma}\label{lem:strong-regularity}
	Let $(\bar{x},\bar{\mu})$ be a solution to the KKT system \eqref{eq:comp-prob-kkt}. Strong regularity of the KKT system \eqref{eq:comp-prob-kkt} at $(\bar{x},\bar{\mu})$ is equivalent to that the mapping $R$ defined in \eqref{eq:kkt-generalized-equation} is a locally Lipschitz homeomorphism near $(\bar{x},\bar{\mu})$.
\end{lemma}
\begin{proof}
	Denote
	\begin{equation*}
		\widetilde{R}(x,\mu):=\left[\begin{array}{c}
			\langle F''(\bar{x})(x-\bar{x}),\bar{\mu}\rangle+F'(\bar{x})^{*}(\mu-\bar{\mu})\\
			\mu-\operatorname{Prox}_{g^{*}}(F(\bar{x})+F'(\bar{x})(x-\bar{x})+\mu)
		\end{array}\right].
	\end{equation*} 
	We first show  that strong regularity of the KKT system \eqref{eq:comp-prob-kkt} at $(\bar{x},\bar{\mu})$ is equivalent to  $\widetilde{R}$ being a locally Lipschitz homeomorphism near $(\bar{x},\bar{\mu})$.
	
	Suppose that $\widetilde{R}$ is a locally Lipschitz homeomorphism near $(\bar{x},\bar{\mu})$. There exists an open neighborhood $\mathcal{N}$ of $(\bar{x},\bar{\mu})$ such that the restriction  $\widetilde{R}_{\mathcal{N}}$ is bijective, $\widetilde{R}(\mathcal{N})$ is an open neighborhood of the origin, and the inverse mapping $(\widetilde{R}|_{\mathcal{N}})^{-1}:\widetilde{R}(\mathcal{N})\rightarrow\mathcal{N}$ is Lipschitz continuous. Consequently, for every $\tilde{\delta}\in\widetilde{R}(\mathcal{N})$, there exists a unique $\tilde{z}(\tilde{\delta})\in\mathcal{N}$ satisfying $\widetilde{R}(\tilde{z})=\tilde{\delta}$.
	Let $\delta=(\delta_{1},\delta_{2})\in\mathcal{R}^{n}\times\mathcal{R}^{m}$ and $z(\delta)$ be a solution of the following generalized equation
	\begin{equation}\label{eq:comp-prob-kkt-linearize}
		\delta\in\left[\begin{array}{c} \langle F''(\bar{x})(x-\bar{x}),\bar{\mu}\rangle+ F'(\bar{x})^{*}(\mu-\bar{\mu})\\-F(\bar{x})-F'(\bar{x})(x-\bar{x})\end{array}\right]+\left[\begin{array}{c}0\\\partial g^{*}(\mu)\end{array}\right].
	\end{equation}
	It is evident that the generalized equation   \eqref{eq:comp-prob-kkt-linearize} shares the same solution set as the following nonsmooth equation
	\begin{equation*}
		\left[\begin{array}{c}
			\langle F''(\bar{x})(x-\bar{x}),\bar{\mu}\rangle+ F'(\bar{x})^{*}((\mu+\delta_{2})-\bar{\mu})\\
			(\mu+\delta_{2})-\operatorname{Prox}_{g^{*}}(F(\bar{x})+F'(\bar{x})(x-\bar{x})+(\mu+\delta_{2}))
		\end{array}\right]=\left[\begin{array}{c}
		\delta_{1}+F'(\bar{x})^{*}\delta_{2}\\\delta_{2}
	\end{array}\right].
	\end{equation*} 
	This implies that for any $
	\delta\in\mathcal{U}:=\frac{1}{1+\|F'(\bar{x})\|}\widetilde{R}(\mathcal{N})$,
	the solution mapping $z$ of \eqref{eq:comp-prob-kkt-linearize} can be expressed as 
	\begin{equation*}
		z(\delta)=\tilde{z}(\delta_{1}+F'(\bar{x})^{*}\delta_{2},\delta_{2})-\left[\begin{array}{c}0\\\delta_{2}\end{array}\right]
	\end{equation*} 
	and it is also Lipschitz continuous on $\mathcal{U}$. Conversely, assuming the KKT system \eqref{eq:comp-prob-kkt} is strongly regular at $(\bar{x},\bar{\mu})$, there exist neighborhoods $\mathcal{U}$ of the origin and $\mathcal{N}$ of $(\bar{x},\bar{\mu})$ such that for any $\delta\in\mathcal{N}$, $z(\delta)\in\mathcal{N}$ is the unique solution to \eqref{eq:comp-prob-kkt-linearize}. For any 
	$\tilde{\delta}\in\mathcal{U}:=\frac{1}{1+\|F'(\bar{x})\|}\mathcal{N}$, 
	the nonsmooth equation $\widetilde{R}(z)=\tilde{\delta}$ admits a unique solution as 
	\begin{equation*}
	\tilde{z}(\tilde{\delta})=z(\tilde{\delta}_{1}-F'(\bar{x})^{*}\tilde{\delta}_{2},\tilde{\delta}_{2})+\left[\begin{array}{c}0\\\tilde{\delta}_{2}\end{array}\right], 
	\end{equation*} 
	which implies that $\tilde{z}(\cdot)$ is Lipschitz continuous on $\mathcal{U}$ and the mapping $\widetilde{R}$ is a locally Lipschitz homeomorphism near $(\bar{x},\bar{\mu})$. The desired result follows directly from \cite[Theorem~3.1]{Kummer1991} that the mapping $\widetilde{R}$ is a locally Lipschitz homeomorphism near $(\bar{x},\bar{\mu})$ if and only if the mapping $R$ possesses the same property and this completes the proof.
\end{proof}

To establish our primary result, we provide detailed elaboration on the domain of $\Gamma_{g}(F(\bar{x}),\bar{\mu})(F'(\bar{x})(\cdot))$ under SRCQ for problem $(P)$.

\begin{lemma}\label{prop-9}
	Let $(\bar{x},\bar{\mu})$ be a solution to the KKT system \eqref{eq:comp-prob-kkt}. Suppose that the function $g$ is a proper l.s.c. convex function, $g$ is parabolically epi-differentiable at F($\bar{x})$ for every $w\in\{d\, |\, g^{\downarrow}_{-}(F(\bar{x}),d)=\langle \bar{\mu},d\rangle\}$, parabolically regular at $F(\bar{x})$ for $\bar{\mu}$, and satisfies Assumption \ref{assumption-2} at $F(\bar{x})$ for $\bar{\mu}$. If SRCQ \eqref{eq:srcq} is valid, then we have
	\begin{equation*}
		\operatorname{dom}\Gamma_{g}(F(\bar{x}),\bar{\mu})(F'(\bar{x})(\cdot))=\{d\,|\, F'(\bar{x})d\in\operatorname{aff}\mathcal{N}_{\partial g(F(\bar{x}))}(\bar{\mu})\}=\operatorname{aff}\mathcal{C}(\bar{x}).
	\end{equation*}
\end{lemma}
\begin{proof}
	By Lemma \ref{lemma-relation-range-partial}, the first desired equality follows immediately.
	The inclusion  $\operatorname{aff}\mathcal{C}(\bar{x})\subseteq\{d\,|\, F'(\bar{x})d\in\operatorname{aff}\mathcal{N}_{\partial g(F(\bar{x}))}(\bar{\mu})\}$ holds naturally. We now establish the reverse inclusion
	\begin{equation}\label{eq:prop9-result}
		\{d\,|\, F'(\bar{x})d\in\operatorname{aff}\mathcal{N}_{\partial g(F(\bar{x}))}(\bar{\mu})\}\subseteq\operatorname{aff}\mathcal{C}(\bar{x}).
	\end{equation}
	Let $\bar{d}$ be an arbitrary element of $\{d\,|\, F'(\bar{x})d\in\operatorname{aff}\mathcal{N}_{\partial g(F(\bar{x}))}(\bar{\mu})\}$ such that $F'(\bar{x})\bar{d}=\bar{d}_{1}-\bar{d}_{2}$, where $\bar{d}_{1},\bar{d}_{2}\in\mathcal{N}_{\partial g(F(\bar{x}))}(\bar{\mu})$. The convexity of $\mathcal{N}_{\partial g(F(\bar{x}))}(\bar{\mu})$ confirms that  $\bar{d}_{1}+d\in\mathcal{N}_{\partial g(F(\bar{x}))}(\bar{\mu})$ holds for any $d\in\mathcal{N}_{\partial g(F(\bar{x}))}(\bar{\mu})$. By SRCQ \eqref{eq:srcq}, there exist $d_{1}\in\mathcal{R}^{n}$ and $d_{2}\in\mathcal{N}_{\partial g(F(\bar{x}))}(\bar{\mu})$ such that 
	$F'(\bar{x})(-d_{1})+d_{2}=-\bar{d}_{1}-d$. This implies $F'(\bar{x})d_{1}=\bar{d}_{1}+d+d_{2}\in\mathcal{N}_{\partial g(F(\bar{x}))}(\bar{\mu})$ 
	and consequently  $d_{1}\in\mathcal{C}(\bar{x})$ and $F'(\bar{x})(d_{1}-\bar{d})=d+d_{2}+\bar{d}_{2}\in\mathcal{N}_{\partial g(F(\bar{x}))}(\bar{\mu})$, which shows that $d_{1}-\bar{d}\in\mathcal{C}(\bar{x})$ and $\bar{d}=d_{1}-(d_{1}-\bar{d})\in\operatorname{aff}\mathcal{C}(\bar{x})$. The required inclusion \eqref{eq:prop9-result} follows and this completes the proof.
\end{proof}

Lemma \ref{prop-9} demonstrates that
under SRCQ \eqref{eq:srcq}, SSOSC of problem $(P)$ admits an equivalent formulation as
\begin{equation*}
	\langle \bar{\mu},F''(\bar{x})(d,d)\rangle +\Gamma_{g}(F(\bar{x}),\bar{\mu})(F'(\bar{x})d)>0,\quad\forall\, d\in\operatorname{aff}\mathcal{C}(\bar{x})\setminus\{0\}.
\end{equation*}

We now present the first main theorem that characterizes the relationships among the following three conditions: SSOSC combined with the constraint nondegeneracy condition, the nonsingularity of Clarke's generalized Jacobian of the KKT system, and strong regularity of the KKT point.
\begin{theorem}\label{thm:ssosc-strong-regularity}
	Let $g$ be a proper l.s.c. convex function and $(\bar{x},\bar{\mu})$ be a solution to the KKT system \eqref{eq:comp-prob-kkt}. Assume that $g$ is parabolically regular at $F(\bar{x})$ for $\bar{\mu}$, parabolically epi-differentiable at F($\bar{x})$ for every $w\in\{d\, |\, g^{\downarrow}_{-}(F(\bar{x}),d)=\langle \bar{\mu},d\rangle\}$, and satisfies Assumption \ref{assumption-2} at $F(\bar{x})$ for $\bar{\mu}$. Under SSOSC \eqref{eq:comp-prob-SSOSC} and the nondegeneracy condition \eqref{eq:comp-prob-nondegeneracy-cond}, all elements of $\partial R(\bar{x},\bar{\mu})$ are nonsingular. Moreover, if all elements of $\partial R(\bar{x},\bar{\mu})$ are nonsingular, then $(\bar{x},\bar{\mu})$ is a strongly regular solution of the KKT system \eqref{eq:comp-prob-kkt}.
\end{theorem}
\begin{proof}
	Let $E$ be an arbitrary element of $\partial R(\bar{x},\bar{\mu})$ and $(\Delta x,\Delta\mu)\in\mathcal{R}^{n}\times\mathcal{R}^{m}$ satisfy $E(\Delta x,\Delta\mu)=0$. There exists  $V\in\partial\operatorname{Prox}_{g^{*}}(F(\bar{x})+\bar{\mu})$ with $U=I-V\in\partial\operatorname{Prox}_{g}(F(\bar{x})+\bar{\mu})$ such that
	\begin{equation}\label{eq:prop9-1}
		\left[\begin{array}{c}\langle F''(\bar{x})\Delta x,\bar{\mu}\rangle+F'(\bar{x})^{*}\Delta\mu\\
			F'(\bar{x})\Delta x - U(F'(\bar{x})\Delta x+\Delta\mu)\end{array}\right]=0.
	\end{equation}
	Applying Lemma \ref{lemma-Gamma-inequality} and Lemma \ref{prop-9} to the second equality of \eqref{eq:prop9-1}, it follows that
	\begin{equation}\label{eq:prop9-2-1}
		\begin{aligned}
		&F'(\bar{x})\Delta x=U(F'(\bar{x})\Delta x+\Delta\mu)\in\operatorname{aff}\mathcal{N}_{\partial g(F(\bar{x}))}(\bar{\mu}),\\
		&\Gamma_{g}(F(\bar{x}),\bar{\mu})(F'(\bar{x})\Delta x)
		\leq\langle F'(\bar{x})\Delta x,\Delta\mu\rangle.
		\end{aligned}
	\end{equation}
    Since the nondegeneracy condition \eqref{eq:comp-prob-nondegeneracy-cond} implies SRCQ \eqref{eq:srcq} and $\Lambda(\bar{x})=\{\bar{\mu}\}$,  combining Lemma \ref{prop-9}, the first equality of \eqref{eq:prop9-1} and \eqref{eq:prop9-2-1} indicate that $\langle\bar{\mu},F''(\bar{x})(\Delta x,\Delta x)\rangle+\langle\Delta x,F'(\bar{x})^{*}\Delta
	\mu\rangle=0$ and $\Delta x\in\operatorname{aff}\mathcal{C}(\bar{x})$,	$\langle\bar{\mu},F''(\bar{x})(\Delta x,\Delta x)\rangle+\Gamma_{g}(F(\bar{x}),\bar{\mu})(F'(\bar{x})\Delta x)\leq0$. Given that SSOSC \eqref{eq:comp-prob-SSOSC} holds, it derives that $\Delta x=0$ and consequently  \eqref{eq:prop9-1} reduces to
	\begin{equation}\label{eq:prop9-3}
		F'(\bar{x})^{*}\Delta\mu=0,\; U\Delta\mu=0.
	\end{equation}	
	By invoking Lemmas \ref{lemma-Gamma-inequality}, \ref{lemma-relation-range-partial} and Proposition \ref{prop-4}, the results in \eqref{eq:prop9-3} imply that $\Delta \mu\in\operatorname{aff}\{d\, |\, g^{*\downarrow}_{-}(\bar{\mu};d)=\langle F(\bar{x}),d\rangle\}$, $\Gamma_{g^{*}}(\bar{\mu},F(\bar{x}))(\Delta\mu)=0$
	and in addition
	\begin{equation*}
		\Delta\mu\in\{d\, |\, g^{\downarrow}_{-}(F(\bar{x}),d)=-g^{\downarrow}_{-}(F(\bar{x}),-d)\}^{\circ}\cap\operatorname{ker}F'(\bar{x})^{*}.
	\end{equation*}
	By virtue of the nondegeneracy condition specified in  \eqref{eq:comp-prob-nondegeneracy-cond}, it obtains that $\Delta\mu=0$. With this,
	the first segment of the intended result is established.
	
	Furthermore, suppose every element of $\partial R(\bar{x},\bar{\mu})$ is nonsingular, in light of Clarke's inverse function theorem \cite[Theorem~1]{Clarke1976}, $R$ is a locally Lipschitz homeomorphism near $(\bar{x},\bar{\mu})$. As a direct consequence of Lemma \ref{lem:strong-regularity}, the point $(\bar{x},\bar{\mu})$ is a strongly regular solution of the KKT system \eqref{eq:comp-prob-kkt}. With this, the proof is hereby concluded.
\end{proof}
Prior to further development, we establish some auxiliary lemmas that provide deeper insight into SSOSC.
\begin{lemma}\label{lem-equivalent-SSOSC}
	Let $(\bar{x},\bar{\mu})$ be a solution to the KKT system \eqref{eq:comp-prob-kkt}. Suppose that the function $g$ satisfies Assumption \ref{assumption-3} at $F(\bar{x})$ for $\bar{\mu}$. For any $v\in\operatorname{rge}(D^{*}\operatorname{Prox}_g(F(\bar{x})+\bar{\mu}))$, we have
	\begin{equation*}
		\Gamma_g(F(\bar{x}),\bar{\mu})(v)=\min_{d\in D^{*}(\partial g)(F(\bar{x}),\bar{\mu})(v)}\langle v,d\rangle.
	\end{equation*}
\end{lemma}
\begin{proof}
	By \cite[Theorem~13.52]{Rockafellar1998},  it is easy to see that
	\begin{equation*}
		\Gamma_{g}(F(\bar{x}),\bar{\mu})(-v)\leq\min_{{d\in\mathcal{R}^{m}}\atop{-v\in D^{*}\operatorname{Prox}_{g}(F(\bar{x})+\bar{\mu})(d)}}\langle -v,d+v\rangle=\min_{d\in D^{*}(\partial g)(F(\bar{x}),\bar{\mu})(v)}\langle v,d\rangle.
	\end{equation*}
	In order to prove the reverse inequality, combining Assumption \ref{assumption-3} with \eqref{eq:inverse-of-DS&DS*}, we have
	\begin{equation*}
		\mathop{\operatorname{argmin}}_{{d\in\mathcal{R}^{m},-v=Ud}\atop{U\in\partial \operatorname{Prox}_{g}(F(\bar{x})+\bar{\mu})}}\langle -v,d+v\rangle\cap (-v-D^{*}(\partial g)(F(\bar{x}),\bar{\mu})(v))\neq\emptyset
	\end{equation*}
	and there exists $d_{v}=-v-\bar{d}$ with $\bar{d}\in D^{*}(\partial g)(F(\bar{x}),\bar{\mu})(v)$ such that $\Gamma_{g}(F(\bar{x}),\bar{\mu})(v)=\Gamma_{g}(F(\bar{x}),\bar{\mu})(-v)=\langle -v,d_{v}+v\rangle=\langle v,\bar{d}\rangle$. This verifies
	\begin{equation*}
		\Gamma_{g}(F(\bar{x}),\bar{\mu})(v)=\Gamma_{g}(F(\bar{x}),\bar{\mu})(-v)\geq\min_{d\in D^{*}(\partial g)(F(\bar{x}),\bar{\mu})(v)}\langle v,d\rangle
	\end{equation*}
	and thus the desired result follows.
\end{proof}

\begin{remark}
Assumption \ref{assumption-3} plays a crucial role in Lemma \ref{lem-equivalent-SSOSC}. We establish in the appendix a computationally tractable sufficient condition for practical applications. As a concrete demonstration, we also show the validity of this condition for the indicator function of the positive semidefinite cone. 
\end{remark}

\begin{lemma}\label{lemma-equivalent-SSOSC}
	Let $(\bar{x},\bar{\mu})$ be a solution to the KKT system \eqref{eq:comp-prob-kkt}. Suppose that $g$ is a proper l.s.c.  convex $\mathcal{C}^{2}$-cone reducible function and satisfies Assumption \ref{assumption-2} at $F(\bar{x})$ for $\bar{\mu}$. Given any $\bar{v}\in\operatorname{aff}\{d\,|\,g^{\downarrow}_{-}(F(\bar{x}),d)=\langle\bar{\mu},d\rangle\}$ and $\bar{\beta}\in\mathcal{R}$ satisfying $(\bar{v},\bar{\beta})\in\operatorname{dom}(D^{*}\mathcal{N}_{\operatorname{epi}g}((F(\bar{x}),g(F(\bar{x}))),(\bar{\mu},-1)))$, we have
	\begin{equation*}
		\min_{\bar{v}^{*}\in D^{*}(\partial g)(F(\bar{x}),\bar{\mu})(\bar{v})}\langle \bar{v}^{*},\bar{v}\rangle\geq\min_{(\bar{z},\bar{\gamma})\in D^{*}\mathcal{N}_{\operatorname{epi}g}((F(\bar{x}),g(F(\bar{x}))),(\bar{\mu},-1))(\bar{v},\bar{\beta})}\langle(\bar{v},\bar{\beta}), (\bar{z},\bar{\gamma})\rangle.
	\end{equation*}
\end{lemma}
\begin{proof}
	We first show that for any $\bar{v}\in\operatorname{aff}\{d\,|\,g^{\downarrow}_{-}(F(\bar{x}),d)=\langle\bar{\mu},d\rangle\}$, there exists $\bar{\beta}\in\mathcal{R}$ such that $(\bar{v},\bar{\beta})\in\operatorname{dom}(D^{*}\mathcal{N}_{\operatorname{epi}g}((F(\bar{x}),g(F(\bar{x}))),(\bar{\mu},-1)))$. Given $\lambda>0$, $x\in\operatorname{dom}g$, $\alpha:=g(x)$ and $\mu\in\partial g(x)$, Proposition \ref{prop-domd2g},  \cite[Theorem~8.2]{Rockafellar1998} and  \cite[Proposition~2.126]{Bonnans2000} imply that the domain of $D\mathcal{N}_{\operatorname{epi}g}((x,\alpha),\lambda(\mu,-1))$ can be  expressed as
	\begin{equation}\label{eq:dom-graphical-normal-epi-g}
		\begin{aligned}
		\operatorname{dom}(D\mathcal{N}_{\operatorname{epi}g}(x,\alpha),\lambda(\mu,-1)))&=\operatorname{dom}(D(\partial\delta_{\operatorname{epi}g})((x,\alpha),\lambda(\mu,-1)))\\
		&=\{(y,c)\,|\,(\delta_{\operatorname{epi}g})^{\downarrow}_{-}((x,\alpha),(y,c))=\lambda\langle\mu,y\rangle-\lambda c\}\\
		&=\{(y,c)\,|\,g^{\downarrow}_{-}(x,y)=\langle \mu,y\rangle=c\}\\
		&=\{(y,c)\,|\,y\in\operatorname{dom}D(\partial g)(x,\mu),\,c=\langle\mu,y\rangle\}\\
		&:=\mathcal{K}((x,\alpha),\lambda(\mu,-1)).
		\end{aligned}
	\end{equation}	
	According to \eqref{eq:inverse-of-DS&DS*} and \cite[Exercise~10.17]{Rockafellar1998}, the set $-\operatorname{dom}(\widehat{D}^{*}\mathcal{N}_{\operatorname{epi}g}((x,\alpha),\lambda(\mu,-1)))$ satisfies the following relationship
	\begin{align}\label{eq:dom-regular-coderivative-dom-graphical-derivative-epi}
		-\operatorname{dom}(\widehat{D}^{*}\mathcal{N}_{\operatorname{epi}g}((x,\alpha),\lambda(\mu,-1)))&=\operatorname{rge}\widehat{D}^{*}\operatorname{Prox}_{\delta_{\operatorname{epi}g}}(x+\lambda\mu,\alpha-\lambda)\nonumber\\
		&\subseteq\mathcal{N}_{(I+\partial\delta_{\operatorname{epi}g})(x,\alpha)}(x+\lambda\mu,\alpha-\lambda)=\mathcal{N}_{\partial\delta_{\operatorname{epi}g}(x,\alpha)}(\lambda\mu,-\lambda)\nonumber\\
		&=\{(y,c)\,|\,(\delta_{\operatorname{epi}g})^{\downarrow}_{-}((x,\alpha),(y,c))=\lambda\langle\mu,y\rangle-\lambda c\}\nonumber\\
		&=\mathcal{K}((x,\alpha),\lambda(\mu,-1)).
	\end{align}
	Since RCQ holds for the constraint $(x,\gamma)\in\operatorname{epi}g$, the set-valued mapping $M(x,\gamma):=(x,\gamma)-\operatorname{epi}g$ is metrically subregular at $(x,\alpha)$ for $(0,0)$. Moreover,  $\operatorname{epi}g$ is $\mathcal{C}^{2}$-cone reducible at $(x,\alpha)$ with $h$ being the corresponding twice continuously differentiable mapping. By  \cite[Corollary~5.4]{Helmut2019}, for any $(y,c)\in\mathcal{K}((x,\alpha),\lambda(\mu,-1))$, 
	\begin{equation*}
		\begin{aligned}
		D\mathcal{N}_{\operatorname{epi}g}((x,\alpha),\lambda(\mu,-1))(y,c)=\lambda\nabla^{2}\langle( h'(x,\alpha)^{*})^{-1}(\mu,-1),h(\cdot,\cdot)\rangle(x,\alpha)(y,c)&\\
		+\mathcal{N}_{\mathcal{K}((x,\alpha),\lambda(\mu,-1))}(y,c)&.
		\end{aligned}
	\end{equation*}
	Consequently, by \eqref{eq:dom-regular-coderivative-dom-graphical-derivative-epi}, given $(v,\beta)\in\operatorname{dom}(\widehat{D}^{*}\mathcal{N}_{\operatorname{epi}g}((x,\alpha),\lambda(\mu,-1)))\subseteq-\mathcal{K}((x,\alpha),\lambda(\mu,-1))$ and $(v^{*},\beta^{*})\in\lambda\nabla^{2}\langle(h'(x,\alpha)^{*})^{-1}(\mu,-1),h(\cdot,\cdot)\rangle(x,\alpha)(v,\beta)
	+\mathcal{K}((x,\alpha),\lambda(\mu,-1))^{\circ}$, we have
	\begin{equation*}
		\langle((v^{*},\beta^{*}),-(v,\beta)),((y,c),(y^{*},c^{*}))\rangle\leq 0,\;\forall\;(y^{*},c^{*})\in D\mathcal{N}_{\operatorname{epi}g}((x,\alpha),\lambda(\mu,-1))(y,c).
	\end{equation*}
	Therefore 
	\begin{equation*}
		\begin{aligned}
		\lambda\nabla^{2}\langle(h'(x,\alpha)^{*})^{-1}(\mu,-1),h(\cdot,\cdot)\rangle(x,\alpha)(v,\beta)
		+\mathcal{K}((x,\alpha),\lambda(\mu,-1))^{\circ}\\
		\subseteq \widehat{D}^{*}\mathcal{N}_{\operatorname{epi}g}((x,\alpha),\lambda(\mu,-1))(v,\beta).
		\end{aligned}
	\end{equation*}
    Taking into account the definition of limiting/Mordukhovich coderivative together with the outer semicontinuity of limiting/Mordukhovich normal cone mapping (see e.g.,  \cite[Proposition~6.5]{Rockafellar1998}), it yields that
	\begin{equation}\label{eq:relation-normal-gph-normal-epi-g}
		\begin{aligned}
		&\mathcal{N}_{\operatorname{gph}\mathcal{N}_{\operatorname{epi}g}}((F(\bar{x}),g(F(\bar{x}))),(\bar{\mu},-1))\\
		&=\limsup_{((x,g(x)),\lambda(\mu,-1))\rightarrow((F(\bar{x}),g(F(\bar{x}))),(\bar{\mu},-1))}\widehat{\mathcal{N}}_{\operatorname{gph}\mathcal{N}_{\operatorname{epi}g}}((x,g(x)),\lambda(\mu,-1))\\
		&\supseteq\limsup_{((x,g(x)),\lambda(\mu,-1))\rightarrow((F(\bar{x}),g(F(\bar{x}))),(\bar{\mu},-1))}G((x,g(x)),\lambda(\mu,-1)):=\mathcal{G}, 
		\end{aligned}
	\end{equation}
	where
	\begin{equation*}
		\begin{aligned}
		G((x,g(x)),\lambda(\mu,-1))&:=\{((v^{*},\beta^{*}),-(v,\beta))\,|\,(v^{*},\beta^{*})\in\\
		&\lambda\nabla^{2}\langle(h'(x,\alpha)^{*})^{-1}(\mu,-1),h(\cdot,\cdot)\rangle(x,\alpha)(v,\beta)\\
		&
		+\mathcal{K}((x,g(x)),\lambda(\mu,-1))^{\circ},\,(v,\beta)\in\operatorname{dom}(\widehat{D}^{*}\mathcal{N}_{\operatorname{epi}g}((x,\alpha),\lambda(\mu,-1)))\}.
		\end{aligned}
	\end{equation*}
	According to \cite[Exercise~10.17]{Rockafellar1998} and Proposition \ref{prop-domd2g}, it follows that
	\begin{equation}\label{eq:dom-regular-coderivative-dom-graphical-derivative}
		\begin{aligned}
		-\operatorname{dom}\widehat{D}^{*}(\partial g)(x,\mu)&=\operatorname{rge}\widehat{D}^{*}\operatorname{Prox}_{g}(x+\mu)\subseteq\mathcal{N}_{(I+\partial g)(x)}(x+\mu)=\mathcal{N}_{\partial g(x)}(\mu)\\
		&=\{d\,|\,g^{\downarrow}_{-}(x,d)=\langle\mu,d\rangle\}=\operatorname{dom}D(\partial g)(x,\mu).
		\end{aligned}
	\end{equation}
	Assumption \ref{assumption-2}, \eqref{eq:dom-graphical-normal-epi-g}, \eqref{eq:relation-normal-gph-normal-epi-g}  and \eqref{eq:dom-regular-coderivative-dom-graphical-derivative} together with
	\begin{equation*}
		\limsup_{(x,\mu)\rightarrow(F(\bar{x}),\bar{\mu})}\operatorname{dom}\widehat{D}^{*}(\partial g)(x,\mu)=\operatorname{dom}D^{*}(\partial g)(F(\bar{x}),\bar{\mu})=\operatorname{aff}\{d\,|\,g^{\downarrow}_{-}(F(\bar{x}),d)=\langle\bar{\mu},d\rangle\},
	\end{equation*}
	show that
	$
		\operatorname{aff}\{d\,|\,g^{\downarrow}_{-}(F(\bar{x}),d)=\langle\bar{\mu},d\rangle\}\subseteq\limsup\limits_{(x,\mu)\rightarrow(F(\bar{x}),\bar{\mu})}\operatorname{dom}D(\partial g)(x,\mu)
	$
	and for any $\bar{v}\in\operatorname{aff}\{d\,|\,g^{\downarrow}_{-}(F(\bar{x}),d)=\langle\bar{\mu},d\rangle\}$, there exist $(\bar{v}^{*},\bar{\beta}^{*})\in\mathcal{R}^{m+1}$, $\bar{\beta}\in\mathcal{R}$ such that
	$
		(\bar{v},\bar{\beta})\in\operatorname{dom}(D^{*}\mathcal{N}_{\operatorname{epi}g}((F(\bar{x}),g(F(\bar{x}))),(\bar{\mu},-1)))$ and 	$((\bar{v}^{*},\bar{\beta}^{*}),-(\bar{v},\bar{\beta}))\in\mathcal{G}$. 
	Therefore, it follows from  \eqref{eq:relation-normal-gph-normal-epi-g} that for any $\bar{v}\in\operatorname{aff}\{d\,|\,g^{\downarrow}_{-}(F(\bar{x}),d)=\langle\bar{\mu},d\rangle\}$, there exists $\bar{\beta}\in\mathcal{R}$ such that $(\bar{v},\bar{\beta})\in\operatorname{dom}(D^{*}\mathcal{N}_{\operatorname{epi}g}((F(\bar{x}),g(F(\bar{x}))),(\bar{\mu},-1)))$. 
	
	Based on the preceding proof, we have
	\begin{align}\label{eq:inequlity-minimal-coderivative-1}
		&\min_{(\bar{z},\bar{\gamma})\in D^{*}\mathcal{N}_{\operatorname{epi}g}((F(\bar{x}),g(F(\bar{x}))),(\bar{\mu},-1))(\bar{v},\bar{\beta})}\langle(\bar{v},\bar{\beta}), (\bar{z},\bar{\gamma})\rangle\nonumber\\
		&=
		\min_{((\bar{z},\bar{\gamma}),-(\bar{v},\bar{\beta}))\in \mathcal{N}_{\operatorname{gph}\mathcal{N}_{\operatorname{epi}g}}((F(\bar{x}),g(F(\bar{x}))),(\bar{\mu},-1))}\langle(\bar{v},\bar{\beta}), (\bar{z},\bar{\gamma})\rangle\leq\min_{((\bar{v}^{*},\bar{\beta}^{*}),-(\bar{v},\bar{\beta}))\in\mathcal{G}}\langle(\bar{v}^{*},\bar{\beta}^{*}), (\bar{v},\bar{\beta})\rangle\nonumber\\
		&=\liminf_{((x,g(x)),\lambda(\mu,-1))\rightarrow((F(\bar{x}),g(F(\bar{x}))),(\bar{\mu},-1))}\ \min_{((v^{*},\beta^{*}),-(v,\beta))\in G((x,g(x)),\lambda(\mu,-1))}\langle(v^{*},\beta^{*}), (v,\beta)\rangle.
	\end{align}
	For any $\bar{v}^{*}\in D^{*}(\partial g)(F(\bar{x}),\bar{\mu})(\bar{v})$, there exist $(x,\mu)$ and $(v^{*},v)$ with $v^{*}\in\widehat{D}^{*}(\partial g)(x,\mu)(v)$ such that $(x,\mu)\rightarrow(F(\bar{x}),\bar{\mu})$ and $(v^{*},v)\rightarrow(\bar{v}^{*},\bar{v})$. By \eqref{eq:dom-regular-coderivative-dom-graphical-derivative}, Proposition \ref{prop-5} and \cite[Proposition~8.37]{Rockafellar1998}, it follows that for any $-z\in D(\partial g)(x,\mu)(-v)$,
	\begin{equation}\label{eq:inequlity-minimal-coderivative-2}
		-d^{2}g(x,\mu)(-v)=\langle -z,v\rangle\geq\langle v^{*},-v\rangle\rightarrow\langle \bar{v}^{*},-\bar{v}\rangle.
	\end{equation}
	Denote  $\mathcal{T}:=\mathcal{T}_{\operatorname{epi}g}^{2}((x,\alpha),(-v,g^{\downarrow}_{-}(x,-v)))=\{(w,\gamma)\,|\,g^{\downdownarrows}_{-}(x;-v,w)\leq\gamma\}$. Since $g$ is parabolically regular and 
	\begin{equation}\label{eq:inequlity-minimal-coderivative-3}
		\begin{aligned}
		&\min_{((v^{*},\beta^{*}),-(v,\beta))\in G((x,g(x)),\lambda(\mu,-1))}\langle(v^{*},\beta^{*}), (v,\beta)\rangle\\
		&\leq\lambda\langle(h'(x,\alpha)^{*})^{-1}(\mu,-1),h''(x,\alpha)(-(v,\beta),-(v,\beta))\rangle\\
		&=-\lambda\sup_{(w,\gamma)\in\mathcal{T}}\{\langle w,\mu\rangle-\gamma\}=\lambda\inf_{w\in\mathcal{R}^{m}}\{g^{\downdownarrows}_{-}(x;-v,w)-\langle w,\mu\rangle\}=\lambda d^{2}g(x,\mu)(-v),
		\end{aligned}
	\end{equation}
	by virtue of \eqref{eq:inequlity-minimal-coderivative-1}, \eqref{eq:inequlity-minimal-coderivative-2} and \eqref{eq:inequlity-minimal-coderivative-3}, we have 
	\begin{equation*}
		\begin{aligned}
		\min_{(\bar{z},\bar{\gamma})\in D^{*}\mathcal{N}_{\operatorname{epi}g}((F(\bar{x}),g(F(\bar{x}))),(\bar{\mu},-1))(\bar{v},\bar{\beta})}\langle(\bar{v},\bar{\beta}), (\bar{z},\bar{\gamma})\rangle&\leq\liminf_{{(x,\mu)\rightarrow(F(\bar{x}),\bar{\mu})}\atop{v\rightarrow \bar{v}}} d^{2}g(x,\mu)(-v)\\
		&\leq\min_{\bar{v}^{*}\in D^{*}(\partial g)(F(\bar{x}),\bar{\mu})(\bar{v})}\langle \bar{v}^{*},\bar{v}\rangle.
		\end{aligned}
	\end{equation*}
	This completes the proof.
\end{proof}
\begin{remark}\label{remark-equivalent-SSOSC}
    Consequently, as shown by Lemma \ref{lem-equivalent-SSOSC}, Lemma \ref{lemma-equivalent-SSOSC} and \eqref{eq:equiv-Gamma-g}, the second-order subdifferential condition in equation (71) of \cite{MordukhovichNghia2015}, with $\Theta$ replaced by  $\operatorname{epi}g$, implies SSOSC \eqref{eq:comp-prob-SSOSC}, provided that the corresponding Lagrange multiplier set $\Lambda(\bar{x})$ is a singleton and $g$ satisfies Assumptions \ref{assump-1}, \ref{assumption-2} and \ref{assumption-3}.
\end{remark}

Finally, we present the second main theorem of the equivalence between 
SSOSC together with the nondegeneracy condition and strong regularity of the KKT system at the given point when Assumptions \ref{assump-1}, \ref{assumption-2} and \ref{assumption-3} hold.
\begin{theorem}\label{thm:equiv-SSOSC-nondegeneracy-nonconvex}
    Let $(\bar{x},\bar{\mu})$ be a solution to the KKT system \eqref{eq:comp-prob-kkt}. Suppose that the function $g$ satisfies Assumption \ref{assump-1}, Assumptions \ref{assumption-2} and \ref{assumption-3} at $F(\bar{x})$ for $\bar{\mu}$. Then the following statements are equivalent to each other.\\
	{\bf(a)} Both SSOSC \eqref{eq:comp-prob-SSOSC} and the nondegeneracy condition \eqref{eq:comp-prob-nondegeneracy-cond} are satisfied at $\bar{x}$.\\
	{\bf(b)} All the elements of $\partial R(\bar{x},\bar{\mu})$ are nonsingular and $\bar{x}$ is a local minimizer of problem $(P)$.\\
	{\bf(c)} The point $(\bar{x},\bar{\mu})$ is a strongly regular point of the KKT system \eqref{eq:comp-prob-kkt} and $\bar{x}$ is a local minimizer of problem $(P)$.\\
	{\bf(d)} $R$ is a locally Lipschitz homeomorphism near  $(\bar{x},\bar{\mu})$ and $\bar{x}$ is a local minimizer of problem $(P)$.
\end{theorem}
\begin{proof}
	Given $d\in\mathcal{C}(\bar{x})$ and $\mu\in\Lambda(\bar{x})$, it follows from Propositions \ref{prop-domd2g}, \ref{prop-5}, Lemma \ref{lem-equivalent-SSOSC} and \cite[Theorem~13.57]{Rockafellar1998} that $\varphi_{g}(F(\bar
    x),\mu)(F'(\bar{x})d)=d^{2}g(F(\bar{x}),\mu)(F'(\bar{x})d)\geq\Gamma_{g}(F(\bar{x}),\mu)(F'(\bar{x})d)$.
    Therefore, if SSOSC  \eqref{eq:comp-prob-SSOSC} is satisfied, it implies the fulfillment of SOSC \eqref{eq:SOSC}. Consequently, according to \cite[Theorem~6.1]{Mohammadi2020}, $\bar{x}$ can be confirmed as a local minimizer of problem $(P)$. Through the integration of  Proposition \ref{prop-domd2g}, Theorem \ref{thm:ssosc-strong-regularity}, Lemma \ref{lem:strong-regularity}, and Clarke's inverse function theorem \cite{Clarke1976}, the implications  {(a)}$\Rightarrow${(b)}$\Rightarrow${(c)}$\Leftrightarrow${(d)} are derived. Moreover, strong regularity of the KKT system of problem \eqref{general-prob-equivalent} at $((\bar{x},g(F(\bar{x}))),(\bar{\mu},-1))$ for the origin is equivalent to the solution mapping $S$ of the linearized generalized equation 
	\begin{equation*}\label{eq-linearized-kkt}
	\delta'\in\left[\begin{array}{c}
		\langle F''(\bar{x})(x-\bar{x}),\bar{\mu}\rangle+F'(\bar{x})^{*}(\mu-\bar{\mu})\\
		\gamma+1\\
		-(F(\bar{x})+F'(\bar{x})(x-\bar{x}),c)
	\end{array}\right]+\left[\begin{array}{c}
	0\\0\\
	(\mathcal{N}_{\operatorname{epi}g})^{-1}(\mu,\gamma)
\end{array}\right],
	\end{equation*}	
	where $\delta'=(\delta_{1},\delta'_{1},(\delta_{2},\delta'_{2}))$, being single-valued Lipschitz continuous around the origin. For $\delta'$ with $\|\delta'\|$ sufficiently small and  $\delta'_{1}<1$,  it follows from \cite[Theorem~8.9]{Rockafellar1998} that
	\begin{equation*}
		\begin{aligned}
		S(\delta')=\{((x,c),(\mu,\gamma))\;|\; (x,\tilde{\mu})\in Z(\delta_{1}/(1-\delta'_{1}),\delta_{2}),\; \mu=(1-\delta'_{1})\tilde{\mu},\\
	c=g(F(\bar{x})+F'(\bar{x})(x-\bar{x})+\delta_{2})-\delta'_{2},\;\gamma=-1+\delta'_{1}	\},
		\end{aligned}
	\end{equation*} 
	where $Z$ is the solution mapping of the system \eqref{eq:comp-prob-kkt-linearize} with $\bar{\mu
    }$ substituted by $\bar{\mu}/(1-\delta'_{1})$. For a sufficiently small $\delta'_{1}$ in its absolute value, Lemma \ref{lem:strong-regularity} ensures that if $(\bar{x},\bar{\mu})$ is a strongly regular point of the KKT system \eqref{eq:comp-prob-kkt} then  $(\bar{x},\bar{\mu}/(1-\delta'_{1}))$ is also a strongly regular point of the same system.  Therefore, 
	the statement {(c)} is equivalent to strong regularity of the KKT system for problem \eqref{general-prob-equivalent} at the point  $((\bar{x},g(F(\bar{x})),(\bar{\mu},-1))$. As noted in Remark \ref{remark-equivalent-SSOSC}, the second-order subdifferential condition given in equation (71) of \cite{MordukhovichNghia2015} implies SSOSC \eqref{eq:comp-prob-SSOSC}.
	By applying \cite[Theorem~5.6]{MordukhovichNghia2015} to problem \eqref{general-prob-equivalent}, we establish the implication  {(c)}$\Rightarrow${(a)}, relying on the equivalence of  nondegeneracy conditions between problems $(P)$ and \eqref{general-prob-equivalent}. Thus, the desired results are obtained.
\end{proof}
\section{Conclusions and future research}
\label{sec:Conclusion}
In this paper, we present a comprehensive perturbation analysis for a class of composite optimization problems. We construct a kind of second-order variational function and then provide a formal definition of SSOSC applicable to general optimization problems. Our research results establish that SSOSC together with the nondegeneracy condition, the nonsingularity of Clarke's generalized Jacobian of the KKT system, and strong regularity of the KKT point are equivalent under certain mild assumptions. These findings provide a robust foundation for further research in optimization, both from computational and theoretical perspectives.

Our future research will focus on three interconnected directions in composite optimization. Firstly, we aim to identify new classes of nonpolyhedral problems where Assumptions 2.2 and 2.3 can be explicitly characterized using the initial data. Secondly, we will investigate the relaxation of nondegeneracy assumptions currently required for establishing key equivalences in the field. Finally, we plan to develop second-order numerical algorithms with provable local convergence rates that do not rely on nondegeneracy conditions.

\section*{Acknowledgement}
We would like to thank the handling editors Professors Jos\'{e} Carrillo and Fran\c{c}oise Tisseur, and anonymous referees for their valuable suggestions and insightful comments, which allowed us to greatly improve the quality of the manuscript.

\begin{appendices}	
	\section{Sufficient conditions for Assumptions \ref{assumption-2} and \ref{assumption-3}}\label{Appendix-A}
	Verifying the conditions in Assumptions \ref{assumption-2} and \ref{assumption-3} is crucial, as they involve the limiting/Mordukhovich coderivative of the proximal mapping. To facilitate this, we propose alternative sufficient conditions that are more readily verifiable. This appendix is dedicated to establishing these conditions.
	\begin{proposition}\label{prop-assump-sufficient-cond}
		Let $r:\mathcal{R}^{m}\rightarrow (-\infty,+\infty]$ be a proper l.s.c. convex function with $\bar{x}\in\operatorname{dom}r$ and $\bar{u}\in\partial r(\bar{x})$. Suppose that the function $r$ is parabolically epi-differentiable at $\bar{x}$ for every $w\in\{d\, |\, r^{\downarrow}_{-}(\bar{x},d)=\langle \bar{u},d\rangle\}$ and parabolically regular at $\bar{x}$ for $\bar{u}$. Then we have the following statements.\\
		{\bf(a)} The equality 
		\begin{equation}\label{assump-2-first-cond}
			\operatorname{dom}D^{*}(\partial r)(\bar{x},\bar{u})=\operatorname{aff}\{d\,|\,r^{\downarrow}_{-}(\bar{x},d)=\langle\bar{u},d\rangle\}
		\end{equation}
		holds if and only if that $\operatorname{rge}(D^{*}\operatorname{Prox}_{r}(\bar{x}+\bar{u}))$
		is a linear subspace. Moreover, the condition \eqref{assump-2-first-cond} is valid if there exists $W\in\partial_{B}\operatorname{Prox}_{r}(\bar{x}+\bar{u})$ such that $\operatorname{rge}W=\mathop{\cup}\limits_{U\in\partial\operatorname{Prox}_{r}(\bar{x}+\bar{u})}\operatorname{rge}U$. \\
		{\bf(b)}
		The condition $\{d\, |\, 0\in D^{*}\operatorname{Prox}_{r}(\bar{x}+\bar{u})(d)\}=\operatorname{aff}\{d\, |\, 0\in D\operatorname{Prox}_{r}(\bar{x}+\bar{u})(d)\}$ is equivalent to $(\operatorname{lin}(\operatorname{rge}D\operatorname{Prox}_{r}(\bar{x}+\bar{u})))^{\perp}=\mathop{\cup}\limits_{U\in\partial \operatorname{Prox}_{r}(\bar{x}+\bar{u})}\operatorname{NULL}(U)$.\\
		{\bf(c)} For any given $v\in\operatorname{rge}(D^{*}\operatorname{Prox}_{r}(\bar{x}+\bar{u}))$, if there exist $W'\in\partial_{B}\operatorname{Prox}_{r}(\bar{x}+\bar{u})$ and $d'\in\mathcal{R}^{m}$ with $v=W'd'$ such that
		$
		\langle v,d'-v\rangle=\min\limits_{{d\in\mathcal{R}^{m},v=Ud}\atop{U\in\partial \operatorname{Prox}_{r}(\bar{x}+\bar{u})}}\langle v,d-v\rangle$,
		then Assumption \ref{assumption-3} holds.
	\end{proposition}
	\begin{proof}
		By applying \eqref{eq:inverse-of-DS}, \eqref{eq:inverse-of-DS&DS*}, and   \cite[Theorem~13.57]{Rockafellar1998}, it follows that
		\begin{equation*}
		\operatorname{dom}D(\partial r)(\bar{x},\bar{u})=\operatorname{rge}(D\operatorname{Prox}_{r}(\bar{x}+\bar{u}))\subseteq\operatorname{dom}D^{*}(\partial r)(\bar{x},\bar{u})=-\operatorname{rge}(D^{*}\operatorname{Prox}_{r}(\bar{x}+\bar{u})).
		\end{equation*}
		To establish the equivalence in {(a)}, it suffices to ‌prove the necessity‌.
		Combining Proposition \ref{prop-domd2g} with \cite[Theorem~13.52]{Rockafellar1998} yields the following inclusion relationship
	\begin{equation*}
        \begin{aligned}        
			\{d\, |\, r^{\downarrow}_{-}(\bar{x},d)=\langle\bar{u},d\rangle\}&\subseteq-\operatorname{rge}(D^{*}\operatorname{Prox}_{r}(\bar{x}+\bar{u}))\subseteq\mathop{\cup}_{U\in\partial\operatorname{Prox}_{r}(\bar{x}+\bar{u})}\operatorname{rge}U.
        \end{aligned}
	\end{equation*}
		Consequently, if $\operatorname{rge}(D^{*}\operatorname{Prox}_{r}(\bar{x}+\bar{u}))$
		is a linear subspace, the above inclusion relation implies
		\begin{align*}
			\operatorname{aff}\{d\,|\,r^{\downarrow}_{-}(\bar{x},d)=\langle\bar{u},d\rangle\}\subseteq\operatorname{rge}(D^{*}\operatorname{Prox}_{r}(\bar{x}+\bar{u}))\subseteq\operatorname{conv}\mathop{\cup}_{U\in\partial\operatorname{Prox}_{r}(\bar{x}+\bar{u})}\operatorname{rge}U.
		\end{align*}
		Lemma \ref{lemma-range-partial-derivative}  ensures the validity of  \eqref{assump-2-first-cond}. 
		Thus, given the existence of  $W\in\partial_{B}\operatorname{Prox}_{r}(\bar{x}+\bar{u})$ with  $\operatorname{rge}W=\mathop{\cup}\limits_{U\in\partial\operatorname{Prox}_{r}(\bar{x}+\bar{u})}\operatorname{rge}U$, \cite[Proposition~2.4]{Gfrerer2022} indicates that  
		$\mathop{\cup}\limits_{U\in\partial\operatorname{Prox}_{r}(\bar{x}+\bar{u})}\operatorname{rge}U=\operatorname{reg}W\subseteq\operatorname{rge}(D^{*}\operatorname{Prox}_{r}(\bar{x}+\bar{u}))\subseteq\mathop{\cup}\limits_{U\in\partial\operatorname{Prox}_{r}(\bar{x}+\bar{u})}\operatorname{rge}U
		$. Consequently, $\operatorname{rge}(D^{*}\operatorname{Prox}_{r}(\bar{x}+\bar{u}))$ is a linear subspace, which implies \eqref{assump-2-first-cond} from the previous argument.
		
		To prove the statement in part  {(b)}, we invoke  Propositions \ref{prop-2} and \ref{prop-5} to obtain 
		\begin{equation*}
			\operatorname{aff}\{d\, |\, 0\in D\operatorname{Prox}_{r}(\bar{x}+\bar{u})(d)\}=(\operatorname{lin}(\operatorname{rge}(D\operatorname{Prox}_{r}(\bar{x}+\bar{u}))))^{\perp}.
		\end{equation*}
		Combining this with Corollary \ref{corollary-2} establishes the equivalence claimed in {(b)}.
		
		Now we proceed to show the assertion in {(c)}. Given $v\in\operatorname{rge}(D^{*}\operatorname{Prox}_{r}(\bar{x}+\bar{u}))$, it yields from \cite[Proposition~2.4]{Gfrerer2022} that 
		\begin{align*}
			\{d\in\mathcal{R}^{m}\,|\,v=Ud,\,U\in\partial_{B}\operatorname{Prox}_{r}(\bar{x}+\bar{u})\}\subseteq
			\{d\in\mathcal{R}^{m}\,|\,v\in D^{*}\operatorname{Prox}_{r}(\bar{x}+\bar{u})(d)\},
		\end{align*}
		which confirms the validity of  Assumption \ref{assumption-3}.
	\end{proof}
	
	\section{The Validation of Assumptions \ref{assumption-2} and \ref{assumption-3}}\label{Appendix-B}
	Let $X\in S^{m}_{+}$ and $U\in\mathcal{N}_{S^{m}_{+}}(X)$. In this appendix, we demonstrate that Assumptions \ref{assumption-2} and \ref{assumption-3} at $X$ for $U$ are valid when $g$ is the  indicator function of the positive semidefinite cone.
	Denote $S^{m}$ as the space of $m\times m$ symmetric matrices, $S^{m}_{+}$ as
	the cone of all positive semidefinite matrices in $S^{m}$. For any $A,B\in S^{m}$, define Frobenius inner product by
	$\langle A,B\rangle:=\operatorname{Tr}(AB)$, 
	where ``Tr'' denotes the trace of a matrix. Given $X\in S^{m}_{+}$, $U\in\mathcal{N}_{S^{m}_{+}}(X)$ and $A:=X+U$, let $\lambda_{i}(A)$ denote the eigenvalues of $A$ for $i=1,\ldots,m$ and $\Lambda$ be the diagonal matrix of eigenvalues of $A$. The spectral decomposition of $A$ is 
	\begin{equation*}
			A = P\left(\begin{array}{ccc}\Lambda_{\alpha} & 0 & 0 \\						0 & \Lambda_{\beta} & 0\\
			0 & 0 & \Lambda_{\gamma}\end{array}\right)P^{*},
	\end{equation*}
	where $\alpha:=\{i\,|\, \lambda_{i}(A)>0\}$, $\beta:=\{i\,|\,\lambda_{i}(A)=0\}$, $\gamma:=\{i\,|\,\lambda_{i}(A)<0\}$, and $P$ is the corresponding orthogonal matrix of orthonormal eigenvectors.
	\subsection{The validation of Assumption \ref{assumption-2}}
	Since the first equality of Assumption \ref{assumption-2} has been established in \cite[Lemma~6.3]{MordukhovichNghia2015}, we focus solely on proving the second equality.
	For any $W\in\partial\operatorname{\Pi}_{S^{m}_{+}}(A)\ (\textrm{or } W\in\partial_{B}\operatorname{\Pi}_{S^{m}_{+}}(A))$, and for any $D\in S^{m}$ with $V=W(D)$, \cite[Proposition~2.2]{Sun2006} ensures the existence of 
	$Z\in\partial\operatorname{\Pi}_{S^{|\beta|}_{+}}(0)\ (\textrm{or } Z\in\partial_{B}\operatorname{\Pi}_{S^{|\beta|}_{+}}(0))$ such that
	\begin{equation}\label{eq:partial-proximal-indicator-sd}
		V=P\widetilde{V}P^{*},\;\widetilde{V}=\left(\begin{array}{ccc}
			\widetilde{D}_{\alpha\alpha} & \widetilde{D}_{\alpha\beta} & \operatorname{\Sigma}_{\alpha\gamma}\circ\widetilde{D}_{\alpha\gamma}\\
			\widetilde{D}_{\beta\alpha} & Z(\widetilde{D}_{\beta\beta}) & 0\\
			\operatorname{\Sigma}_{\gamma\alpha}\circ\widetilde{D}_{\gamma\alpha} & 0 & 0
		\end{array}\right),
	\end{equation}
	where ``$\circ$'' denotes the Hadamard product, the matrix $\operatorname{\Sigma}$ is defined by
	\begin{equation*}
		\operatorname{\Sigma}_{ij}:=\frac{\max\{\lambda_{i}(A),0\}+\max\{\lambda_{j}(A),0\}}{|\lambda_{i}(A)|+|\lambda_{j}(A)|},\; i,j=1,2,\ldots,m
	\end{equation*}
	with the convention that $0/0:=1$.
	
	By picking $\alpha'=\beta$ in \cite[Lemma~11]{Pang2003}, we can choose $\widehat{Z}\in\partial_{B}\operatorname{\Pi}_{S^{|\beta|}_{+}}(0)$ such that $\widetilde{V}_{\beta\beta}=\widehat{Z}(\widetilde{D}_{\beta\beta})=\widetilde{D}_{\beta\beta}$. Therefore, there exists $W'\in\partial_{B}\operatorname{\Pi}_{S^{m}_{+}}(A)$ such that $\operatorname{rge}W'=\mathop{\cup}\limits_{W\in\partial\operatorname{\Pi}_{S^{m}_{+}}(A)}\operatorname{rge}W$. The first condition of Assumption \ref{assumption-2} holds by applying Proposition \ref{prop-assump-sufficient-cond}(a). 
    
    Now we proceed to prove the validity of the other condition in Assumption \ref{assumption-2}. It follows from \cite[Example~3.140]{Bonnans2000}, Propositions \ref{prop-domd2g}, \ref{prop-5} and \cite{Sun2006} that $\operatorname{\Pi}_{S^{m}_{+}}$ is directionally differentiable at $A$ and
	the explicit expression of the directional derivative $\operatorname{\Pi}_{S^{m}_{+}}'(A,D)$ 
	takes the following form
		\begin{equation*}
			\operatorname{\Pi}_{S^{m}_{+}}'(A,D)=P\left(\begin{array}{ccc}
				\widetilde{D}_{\alpha\alpha} & \widetilde{D}_{\alpha\beta} & \operatorname{\Sigma}_{\alpha\gamma}\circ\widetilde{D}_{\alpha\gamma}\\
				\widetilde{D}_{\beta\alpha} & \operatorname{\Pi}_{S^{|\beta|}_{+}}(\widetilde{D}_{\beta\beta}) & 0\\
				\operatorname{\Sigma}_{\gamma\alpha}\circ\widetilde{D}_{\gamma\alpha} & 0 & 0
			\end{array}\right)P^{*}
		\end{equation*}
	with  $\widetilde{D}=P^{*}DP$ and $D\operatorname{\Pi}_{S^{m}_{+}}(A)(D)=\{\operatorname{\Pi}_{S^{m}_{+}}'(A,D)\}$.
	This ensures that
	\begin{equation*}
	\operatorname{lin}(\operatorname{rge}D\operatorname{\Pi}_{S^{m}_{+}}(A))=\{V=P\widetilde{V}P^{*}\,|\,\widetilde{V}_{\beta\beta}=0,\widetilde{V}_{\beta\gamma}=0,\widetilde{V}_{\gamma\gamma}=0\}.
	\end{equation*}
	Consequently, take the polar operation on both sides and obtain 
	\begin{equation*}
		(\operatorname{lin}(\operatorname{rge}D\operatorname{\Pi}_{S^{m}_{+}}(A)))^{\perp}=\{V=P\widetilde{V}P^{*}\,|\,\widetilde{V}_{\alpha\alpha}=0,\widetilde{V}_{\alpha\beta}=0,\widetilde{V}_{\alpha\gamma}=0\}.
	\end{equation*}
	By virtue of \cite[Lemma~11]{Pang2003}, it confirms the existence of $\widehat{W}$ with
	\begin{equation*}
		\widehat{W}(D)=P\widetilde{W}P^{*},\;\widetilde{W}=\left(\begin{array}{ccc}
			\widetilde{D}_{\alpha\alpha} & \widetilde{D}_{\alpha\beta} & \operatorname{\Sigma}_{\alpha\gamma}\circ\widetilde{D}_{\alpha\gamma}\\
			\widetilde{D}_{\beta\alpha} & 0 & 0\\
			\operatorname{\Sigma}_{\gamma\alpha}\circ\widetilde{D}_{\gamma\alpha} & 0 & 0
		\end{array}\right),\,\forall\, D\in S^{m},
	\end{equation*}
	such that
	$
	\mathop{\cup}\limits_{W\in\partial\operatorname{\Pi}_{S^{m}_{+}}(A)}\operatorname{NULL}(W)=\operatorname{NULL}(\widehat{W})=(\operatorname{lin}(\operatorname{rge}D\operatorname{\Pi}_{S^{m}_{+}}(A)))^{\perp}$.
	Therefore the second equality of Assumption \ref{assumption-2} follows directly by Proposition \ref{prop-assump-sufficient-cond}(b).
	
	\subsection{The validation of Assumption \ref{assumption-3}}	
	Given $X\in S^{m}_{+}$ and $U\in\mathcal{N}_{S^{m}_{+}}(X)$ with $A:=X+U$, we demonstrate that Assumption \ref{assumption-3} holds for $\delta_{S^{m}_{+}}(\cdot)$ at $X$ for $U$. For any $V\in\operatorname{rge}(D^{*}\operatorname{\Pi}_{S^{m}_{+}}(A))$, there exist $W\in\partial\operatorname{\Pi}_{S^{m}_{+}}(A)$ and $D\in S^{m}$ such that $V=W(D)$.
	Applying  Proposition \ref{prop-partial-proximal-mapping} together with  \eqref{eq:partial-proximal-indicator-sd}, we have 
	\begin{equation}\label{eq:sdp-inequality}
    \begin{aligned}
		\langle V,D-V\rangle&=\langle\widetilde{V},\widetilde{D}-\widetilde{V}\rangle=2\langle \widetilde{V}_{\alpha\gamma},\widetilde{D}_{\alpha\gamma}-\widetilde{V}_{\alpha\gamma}\rangle+\langle Z(\widetilde{D}_{\beta\beta}),\widetilde{D}_{\beta\beta}-Z(\widetilde{D}_{\beta\beta})\rangle\\
		&\geq2\langle \widetilde{V}_{\alpha\gamma},\widetilde{D}_{\alpha\gamma}-\widetilde{V}_{\alpha\gamma}\rangle=-2\sum\limits_{i\in\alpha,j\in\gamma}\frac{\lambda_{j}}{\lambda_{i}}(\widetilde{V}_{ij})^{2}.
        \end{aligned}
	\end{equation}
	By selecting $\alpha'=\beta$ in \cite[Lemma~11]{Pang2003} and choosing  $\widehat{Z}\in\partial_{B}\operatorname{\Pi}_{S^{|\beta|}_{+}}(0)$ such that $\widetilde{V}_{\beta\beta}=\widehat{Z}(\widetilde{D}_{\beta\beta})=\widetilde{D}_{\beta\beta}$, the inequality in \eqref{eq:sdp-inequality} becomes an  equality. Therefore for any $V\in\operatorname{rge}(D^{*}\operatorname{\Pi}_{S^{m}_{+}}(A))$, there exist $W'\in\partial_{B}\operatorname{\Pi}_{S^{m}_{+}}(A)$ and $D'\in S^{m}$ satisfying $V=W'(D')$ such that
	\begin{equation*}
	\langle V,D'-V\rangle=\min\limits_{{D\in S^{m},V=U(D)}\atop{U\in\partial \operatorname{\Pi}_{S^{m}_{+}}(A)}}\langle V,D-V\rangle.
	\end{equation*}
	Finally Proposition \ref{prop-assump-sufficient-cond}(c) guarantees the validity of Assumption \ref{assumption-3}.
		
\end{appendices}	

\bibliographystyle{amsplain}
\bibliography{ref}


\end{document}